%% file: unitary_group_homs.tex
\documentclass[12pt]{amsart}

\usepackage[margin=1.5in]{geometry} 

\include{preabmle}

\include{macros}

\include{letterfonts}

\begin{document}
\title{Unitary groups, $K$-theory and traces}
\author{Pawel Sarkowicz}
\email{\href{mailto:psark007@uottawa.ca}{psark007@uottawa.ca}}
\address{Department of Mathematics and Statistics, University of Ottawa, 75 Laurier Ave. East, Ottawa, ON, K1N 6N5 Canada}

  \begin{abstract}
    We show that continuous group homomorphisms between uni\hyp{}tary groups of unital C*-algebras induce maps between spaces of continuous real-valued affine functions on the trace simplices. Under certain $K$-theoretic regularity conditions, these maps can be seen to commute with the pairing between $K_0$ and traces. If the homomorphism is contractive and sends the unit circle to the unit circle, the map between spaces of continuous real-valued affine functions can further be shown to be unital and positive (up to a minus sign).
  \end{abstract}

  \maketitle
  \tableofcontents

  \section{Introduction}

\renewcommand*{\thetheorem}{\arabic{chapter}.\Alph{theorem}}
Unitary groups of C*-algebras have been long studied, and for many classes of operator algebras they form a complete invariant. In \cite{Dye53}, Dye studied the unitary group isomorphism problem between non-atomic W*-algebras, with the assumption of \emph{weak bicontinuity} of the isomorphism. He later showed that the unitary group, this time as an algebraic object, determined the type of a factor \cite{Dye55} (except for type $\I_{2n}$). He showed that such group isomorphisms were the restrictions of a *-isomorphism or a conjugate linear *-isomorphism multiplied by a possibly discontinuous character (\cite[Appendix A]{Booth98} gives exposition). Sakai generalized Dye's results to show that any uniformly continu\hyp{}ous unitary group isomorphism between AW*-factors comes from a *-isomor\hyp{}phism or conjugate-linear *-isomorphism \cite{Sakai55} (see also \cite{Yen56} for general AW*-algebras which have no component of type $I_n$). 

Dye's method was generalized to large classes of real rank zero C*-algebras by Al-Rawashdeh, Booth and Giordano in \cite{AlBoothGiordano12}, where they applied the method to obtain induced maps between $K$-theory, with a general linear variant being done by Giordano and Sierakowski in \cite{GiordanoSierakowski16}. The stably finite and purely infinite cases were handled separately. The unital, simple AH-algebras of slow dimension growth and of real rank zero were classified by the topological group isomorphism class of their unitary groups (or general linear groups), and the unital, simple, purely infinite UCT algebras were classified via the algebraic isomorphism classes of their unitary groups (or general linear groups). These results made use of the abundance of projections in real rank zero C*-algebras (at least to show there were isomorphic $K_0$-groups), and made use of the Dadarlat-Elliott-Gong \cite{Dadarlat95,Gong97} and Kirchberg-Phillips \cite{Phillips00} classification theorems respectively (see Theorems 3.3.1 and 8.4.1 of \cite{RordamBook} for each respective case).

In \cite{Paterson83}, it was proven by Paterson that two unital C*-algebras are iso\hyp{}morphic if and only there is an isometric isomorphism of the unitary groups which acts as the identity on the circle. In a similar vein, the metric structure of the unitary group has also played a role in determining the Jordan *-algebra structure on C*-algebras. In \cite{HatoriMolnar14}, Hatori and Moln{\'a}r showed that two unital C*-algebras are Jordan *-isomorphic if and only if their unitary groups are isometric as metric spaces, not taking into account any algebraic structure.

Chand and Robert have shown in \cite{ChandRobert23} that if $A$ and $B$ are prime traceless C*-algebras with full square zero elements such that $U^0(A)$, the subgroup of unitaries which are path connected to the identity, is algebraically isomorphic to $U^0(B)$, then $A$ is either isomorphic or anti-isomorphic to $B$. In fact, the group isomorphism is the restriction of a *-isomorphism or anti-*-isomorphism which follows from the fact that unitary groups associated to these C*-algebras have certain automatic continuity properties that allow one to use character\hyp{}izations of \emph{commutativity preserving maps} \cite{Bresar93} (see \cite{AraMathieubook}). Chand and Robert also show that if $A$ is a unital separable C*-algebra with at least one tracial state, then $U^0(A)$ admits discontinuous automorphisms. Thus the existence of traces is an obstruction to classification via algebraic structure on the unitary groups -- at least an obstruction to unitary group homomorphisms being the restrictions of *-homomorphisms or anti-*-homomorphisms.

In this paper, we show that uniformly continuous unitary group homo\hyp{}morphisms yield maps between traces which have several desirable $K$-theoretic properties -- especially under stricter continuity assumptions. Namely, that the homomorphism sends the circle to the circle and is contractive, which would be automatic if it had a lift to a *-homomorphism or conjugate-linear *-homomorphism.

We state our main results. Recall that if $A$ is a unital C*-algebra, $T(A)$ denotes the simplex of tracial states and $\Aff T(A)$ is the real Banach space of continuous affine functions $T(A) \to \bR$. 

\begin{result}[Corollary \ref{cor:pairing}]
Let $A,B$ be unital C*-algebras. If $\theta: U^0(A) \to U^0(B)$ is a uniformly continuous group homomorphism, then there exists a bounded $\bR$-linear map $\Lambda_\theta: \Aff T(A) \to \Aff T(B)$ such that
\begin{equation}
\begin{tikzcd}
\pi_1(U^0(A)) \arrow[r, "\tD^1_A"] \arrow[d, "\pi_1(\theta)"'] & \Aff T(A) \arrow[d, "\Lambda_\theta"] \\
\pi_1(U^0(B)) \arrow[r, "\tD^1_B"']                            & \Aff T(B)                     
\end{tikzcd}
\end{equation}
commutes.
\end{result}

Here $\pi_1(\theta)$ is the map between fundamental groups induced by $\theta$, and, for a C*-algebra $A$, $\tD^1_A$ is the \emph{pre-determinant} (used in the definition of the de la Harpe-Skandalis determinant associated to the universal trace) that takes a piece-wise smooth path in $U^0(A)$
beginning at the unit to an element of $\Aff T(A)$. See Section \ref{section:dlHS-det} for details.

Recall that the $K_0$-group of a unital C*-algebra can be identified with the fundamental group $\pi_1(U_{\infty}^0(A))$. Restricting to C*-algebras with sufficient $K_0$-regularity -- by this we mean C*-algebras whose $K_0$-group can be realized as loops in the connected component of its unitary group -- we get a map between $K_0$-groups and a map between spaces of continuous real-valued affine functions on the trace simplex which commute with the pairing.

\begin{resultcor}[Corollary \ref{cor:pairing}]
Let $A,B$ be unital C*-algebras such that the canonical maps
\begin{equation}\label{eq:cor3-hyp-iso}
\pi_1(U^0(A)) \to K_0(A) \text{ and }\pi_1(U^0(B)) \to K_0(B)
\end{equation}
are isomorphisms. If $\theta: U^0(A) \to U^0(B)$ is a continuous group homomorphism then there exists a bounded linear map $\Lambda_\theta: \Aff T(A) \to \Aff T(B)$ such that
\begin{equation}
\begin{tikzcd}
K_0(A) \arrow[r, "\rho_A"] \arrow[d, "K_0(\theta)"'] & \Aff T(A) \arrow[d, "\Lambda_\theta"] \\
K_0(B) \arrow[r, "\rho_B"']                            & \Aff T(B)                     
\end{tikzcd}
\end{equation}
commutes, where $K_0(\theta)$ is the map induced by $\pi_1(\theta)$ together with the isomorphisms of (\ref{eq:cor3-hyp-iso}).
\end{resultcor}

C*-algebras satisfying the above hypothesis are quite common -- for example C*-algebras having stable rank one \cite{Rieffel87} or that are $\cZ$-stable \cite{Jiang97} have this property. Viewing $\Aff T(A)$ and $\Aff T(B)$ as partially ordered real Banach spaces (under the uniform norm) with order units, it is not however true that $\Lambda_\theta$ is unital or positive (see Example \ref{example:nonpositive-example}). This is remedied by adding stricter continuity assumptions on the homomorphism $\theta$ (and possibly by replacing $\Lambda_\theta$ with $-\Lambda_\theta$).

When $\theta: U(A) \to U(B)$ is contractive, injective and sends the circle to the circle, then we show (Lemma \ref{lem:tildeS-isometric}) that either $\Lambda_\theta$ or $-\Lambda_\theta$ is unital and positive, and therefore $\theta$ induces a map between $K$-theory and traces in such a manner that respects the pairing (which in turn gives a map between Elliott invariants for certain simple C*-algebras). As a consequence, we can identify certain unitary subgroups with C*-subalgebras by using $K$-theoretic classification of embeddings \cite{CGSTW23}. 

\begin{result}[Corollary \ref{cor:ugh-subgroup-subalgebra}]
Let $A$ be a unital, separable, simple, nuclear C*-algebra satisfying the UCT which is either $\cZ$-stable or has stable rank one, and $B$ be a unital, separable, simple, nuclear $\cZ$-stable C*-algebra. If there is a contractive injective group homomorphism $U(A) \to U(B)$ which maps the circle to the circle, then there is a unital embedding $A \into B$.
\end{result}

This paper is structured as follows. In Section \ref{section:ugh-ctugh-and-traces} we use a continuous unitary group homomorphism to construct a map between spaces of continuous affine functions on the trace simplices, and use the de la Harpe-Skandalis determinant to show that this map has several desirable properties with respect to the map induced on the fundamental groups of the unitary groups. In Section \ref{section:ugh-order-on-aff-classification} we discuss how our map between spaces of affine functions respects or flips the order under certain continuity assumptions on the unitary group homomorphism. In Section \ref{section:ugh-general-linear-variants} we discuss a slight general linear variant. We finish in Section \ref{section:ugh-final-remarks} with some open questions.

  \addtocontents{toc}{\protect\setcounter{tocdepth}{0}}
  \section*{Acknowledgements}
  Many thanks to my PhD supervisors Thierry Giordano and Aaron Tikuisis for many helpful discussions. Thanks to the authors of \cite{CGSTW23} for sharing a draft of their paper. Finally, thanks to the referee for asking for clarification on the general linear variant, which led me to consider a counter-example to $\bC$-linearity. 

  \addtocontents{toc}{\protect\setcounter{tocdepth}{1}}

\renewcommand*{\thetheorem}{\arabic{section}.\arabic{theorem}}

\section{Preliminaries and notation}

\subsection{Notation} 
For a group $G$, we denote by $DG$ the derived subgroup of $G$, i.e.,
\begin{equation} DG := \langle ghg^{-1}h^{-1} \mid g,h \in G \rangle. \end{equation}
If $G$ has an underlying topology, we denote by $CG$ the closure of $DG$ and $G^0$ the connected component of the identity. 

For a unital C*-algebra $A$, $U(A)$ denotes the unitary group of $A$, while $U^0(A)$ denotes the connected component of the identity in $U(A)$. For $n \in \bN$, we write $U_n(A) := U(M_n(A))$, $U_n^0(A) := U^0(M_n(A))$, and we set
\begin{equation} U_{\infty}(A) := \dlim\, U_n(A), \end{equation}
to be the inductive limit with connecting maps $U_n(A) \ni u \mapsto u \oplus 1 \in U_{n+1}(A)$. 
This makes $U_\infty(A)$ both a topological space (with the inductive limit topology) and a group.\footnote{As pointed out in footnote 58 of \cite{CGSTW23}, $U_{\infty}(A)$ is not in general a topological group since multiplication is not in general jointly continuous in this topology.}
We have general linear analogues by replacing $U$ with $GL$, where $GL(A)$ denotes the group of invertible elements of $A$. 
Similarly, we define $M_{\infty}(A) = \dlim \, M_n(A)$ (as an algebraic direct limit) with connecting maps $x \mapsto x \oplus 0$. If $E$ is real Banach space and $\tau: A_{sa} \to E$ is a linear map that is tracial (i.e., $\tau(a^*a) = \tau(aa^*)$ for all $a \in A$), we extend this canonically to $(M_{\infty}(A))_{sa}$ by setting $\tau((a_{ij})) := \sum_i \tau(a_{ii})$ for $(a_{ij})\in (M_n(A)_{sa}$. 

We write $\pi_0(X)$ for the space of connected components of a topological space $X$, and $\pi_1(X)$ for the fundamental group of $X$ with distinguished base point. In our case, we will usually have $X = U_n(A)$ or $X = U_n^0(A)$, for $n \in \bN \cup \{\infty\}$, with the base point being the unit. 

For a unital C*-algebra $A$, we let $K_0(A),K_1(A)$ be the topological $K$-groups of $A$. 
We will often use the identification of $K_0(A)$ with the fundamental group $\pi_1(U_{\infty}^0(A))$ (see for example \cite[Chapter 11.4]{RordamKBook}).
The set of tracial states on $A$ will be denoted $T(A)$, which is a Choquet simplex (\cite[Theorem 3.1.18]{Sakaibook}), and we denote by $\Aff T(A)$ the set of continuous affine functions $T(A) \to \bR$, which is an interpolation group with order unit (see \cite[Chapter 11]{Goodearlbook}).
For unital $A$, the pairing map $\rho_A: K_0(A) \to \Aff T(A)$ is defined as follows: if $x \in K_0(A)$, we can write $x = [p] - [q]$ where $p,q \in M_n(A)$ are projections, and then
\begin{equation}
\rho_A(x)(\tau) := \tau (p - q), \ \ \ \tau \in T(A).
\end{equation}

\subsection{The de la Harpe--Skandalis determinant and Thomsen's variant}\label{section:dlHS-det}

We recall the definition of the unitary variant of the de la Harpe--Skandalis determinant \cite{dlHS84a} (see \cite{dlHarpe13} for a more in-depth exposition). By a bounded trace we mean a bounded linear map $\tau: A_{sa} \to E$, where $E$ is a real Banach space, such that $\tau(a^*a) = \tau(aa^*)$ for all $a \in A$. For $n \in \bN \cup \{\infty\}$, a bounded trace $\tau: A_{sa} \to E$, and a piece-wise smooth path $\xi: [0,1] \to U_n(A)$, set
\begin{equation} 
\label{eq:predetDefn}
\tD_\tau^n(\xi) := \int_0^1 \tau\left(\frac{1}{2\pi i}\xi'(t)\xi(t)^{-1}\right)dt,
\end{equation}
where this integral is just the Riemann integral in $E$.\footnote{Note that $\xi'(t)\xi(t)$ is skew-adjoint by \cite[Proposition 1.4]{Phillips92} so that $\frac{1}{2\pi i}\xi'(t)\xi(t)^{-1}$ is self-adjoint.}
 We state the unitary variant of \cite[Lemme 1]{dlHS84a}.

\begin{prop}\label{prop:detFacts}
Let $\tau: A_{sa} \to E$ be a bounded trace and $n \in \bN \cup \{\infty\}$. The map $\tD_\tau^n$, which takes a piece-wise smooth path in $U_n^0(A)$ to an element in $E$, has the following four properties:
  \begin{enumerate}
    \item it takes pointwise products to sums: if $\xi_1,\xi_2$ are two piece-wise smooth paths, then
    \begin{equation} \tD_\tau^n(\xi_1\xi_2) = \tD_\tau^n(\xi_1) + \tD_\tau^n(\xi_2), \end{equation}
    where $\xi_1\xi_2$ is the piece-wise smooth path $t \mapsto \xi_1(t)\xi_2(t)$ from $\xi_1(0)\xi_2(0)$ to $\xi_1(1)\xi_2(1)$;
  \item if $\|\xi(t) - 1\| < 1$ for all $t \in [0,1]$, then \begin{equation} 2\pi i \tD_\tau^n(\xi) = \tau\big(\log\xi(1) - \log\xi(0)\big); \end{equation}
    \item $\tD_\tau^n(\xi)$ depends only on the continuous homotopy class of $\xi$;
    \item if $p \in M_n(A)$ is a projection, then the path $\xi_p: [0,1] \to U_n^0(A)$ given by $\xi_p(t) := pe^{2\pi i t} + (1-p)$ satisfies $\tD_\tau^n(p) = \tau(p)$. \label{predetprojection}
  \end{enumerate}
\end{prop}

  The de la Harpe--Skandalis determinant associated to $\tau$ (at the $n^{\text{th}}$ level) is then the map
  \begin{equation} \Delta_\tau^n: U_{\infty}^0(A) \to E/\tD_\tau^n(\pi_1(U_n^0(A))) \end{equation}
  given by $\Delta_\tau^n(x) := [\tD_\tau^n(\xi_x)]$ where $\xi_x$ is any piece-wise smooth path $\xi_x: [0,1] \to U_n^0(A)$ from 1 to $x$. This is a well-defined group homomorphism (using Proposition \ref{prop:detFacts}) to an abelian group and therefore factors through the derived group, i.e., $DU_n^0(A) \subseteq \ker \Delta_\tau^n$. For the $n = \infty$ case, we just write $\tD_\tau$ and $\Delta_\tau$ for $\tD_\tau^{\infty}$ and $\Delta_\tau^{\infty}$ respectively.
  
  We will often be interested in the \emph{universal trace} $\Tr_A: A_{sa} \to \Aff T(A)$, which is  given by $\Tr_A(a) := \hat{a}$, where $\hat{a} \in \Aff T(A)$ is the function given by $\hat{a}(\tau) := \tau(a)$ for $\tau \in T(A)$. We note that in this case, for $[x] \in K_0(A)$, we have that $\Tr(x) = \rho_A([x])$.  When considering the universal trace, we will write $\Delta^n$ and $\Delta$ for $\Delta_{\Tr}^n$ and $\Delta_{\Tr}^{\infty}$ respectively. If the C*-algebra needs to be specified, we write $\Delta_A^n$ or $\Delta_A$.
  
      \begin{prop}\label{rem:cts-to-piece-wise-smooth}\mbox{}
Let $n \in \bN \cup \{\infty\}$. Every continuous path $\xi: [0,1] \to U_n(A)$ is homotopic to a piece-wise smooth path {\normalfont (}even a piece-wise smooth exponential path if we are in $U_n^0(A)${\normalfont )}. Moreover, there exists $a \in A_{sa}$ such that $\tD_\omega^n(\xi) = \omega(a)$ whenever $\omega: A_{sa} \to E$ is a bounded trace.

In particular, as $\tD^n$ is homotopy-invariant, it makes sense to apply $\tD^n$ to any continuous path. 
\begin{proof}
This is essentially \cite[Lemme 3]{dlHS84a}. Take a continuous path $\xi: [0,1] \to U_n(A)$ and choose $k$ such that
    \begin{equation}
    \left\|\xi(\frac{j-1}{k})^{-1}\xi(\frac{j}{k}) - 1\right\| < 1 \text{ for all } j=1,\dots,k.
    \end{equation}
    Then taking
    \begin{equation}
    a_j := \frac{1}{2\pi i}\log \left(\xi(\frac{j-1}{k})^{-1}\xi(\frac{j}{k})\right), j=1,\dots,k,
    \end{equation}
     $\xi$ will be homotopic to the path
    \begin{equation}
    \eta(t) = \xi\left(\frac{j-1}{k}\right)e^{2\pi i(kt - j+1)a_j}, t \in \left[\frac{j-1}{k},\frac{j}{k}\right], j=1,\dots,k.
    \end{equation}
            We note that $\tD_\tau^n(\eta) = \sum_{j=1}^k \tau(a_j)$. Indeed, for simplicity denote by 
 \begin{equation}
 X_j := \xi\left(\frac{j-1}{k}\right)\text{ and }Y_j := e^{2\pi i(kt - j + 1)a_j}. 
 \end{equation}
 Then
\begin{equation}\label{eq:xi-to-eta-ajs}
\begin{split}
 \tD_\omega^n(\xi) &=  \tD_\omega^n(\eta) \\
 &=\sum_{j=1}^k \int_{\frac{j-1}{k}}^{\frac{j}{k}}\tau\left(\frac{1}{2\pi i}X_j 2\pi i k a_jY_jY_j^*X_j^*\right) dt \\
 &= \sum_{j=1}^k \int_{\frac{j-1}{k}}^{\frac{j}{k}}k\tau(a_j) dt \\
 &= \sum_{j=1}^k \tau(a_j). \\
 \end{split}
\end{equation}
If we take $a := \tr(a')$, where $a' = \sum_{j=1}^k a_j$ (here $\tr: M_{\infty}(A) \to A$ is the unnormalized trace), we see that $\tD_\omega^n(\xi) = \omega(a)$.
 \end{proof}
    \end{prop}
  
  Let $A_0$ consist of elements $a \in A_{sa}$ satisfying $\tau(a) = 0$ for all $\tau \in T(A)$.
  This is a norm-closed real subspace of $A_{sa}$ such that $A_0 \subseteq \ov{[A,A]}$, and there is an isometric identification $A_{sa}/A_0 \simeq \Aff T(A)$ sending an element $[a]$ to $\widehat{a}$.
  Indeed, it is not hard to see that the map $A_{sa}/A_0 \to \Aff T(A)$ given by $[a] \mapsto \hat{a}$ is a well-defined $\bR$-linear map.
  Moreover \cite[Theorem 2.9]{CuntzPedersen79}, together with a convexity argument, gives that this is isometric identification.
  To see that we have all the real-valued affine functions, we note that the image of this map contains constant functions and separates points, so \cite[Corollary 7.4]{Goodearlbook} gives that the image is dense and therefore all of $\Aff T(A)$ (since this is an isometry).
  We freely identify $A_{sa}/A_0$ with $\Aff T(A)$.

      \subsection{Thomsen's variant} Thomsen's variant of the de la Harpe--Skandalis determinant is the Hausdorffized version, taking into account the closure of the image of the homotopy groups. For a bounded trace $\tau: A_{sa} \to E$, we consider the map
      \begin{equation} \bar{\Delta}_\tau^n: U_n^0(A) \to E/\ov{\tD_\tau^n(\pi_1(U_n^0(A)))}, \end{equation}
      given by $\bar{\Delta}_\tau^n(x) := [\tD_\tau^n(\xi_x)]$ where $\xi_x: [0,1] \to U_n^0(A)$ is any piece-wise smooth path from 1 to $x \in U_n^0(A)$.
This is similar to the map $\Delta_\tau^n$, except the codomain is now the quotient by the closure of the image of the fundamental group under the pre-determinant (i.e., the Hausdorffization of the codomain).
When considering the universal trace, we just write $\ov\Delta^n$ for $\ov\Delta_{\Tr}^n$ and $\ov\Delta$ for $\ov\Delta_{\Tr}^{\infty}$. If the C*-algebra needs to be specified, we write $\ov\Delta_A^n$ or $\ov\Delta_A$.

If one considers the universal trace, the kernel of $\ov\Delta^n$ can be identified with $CU_n^0(A)$ (where the closure is taken with respect to the inductive limit topology in the $n = \infty$ case). 

\begin{lemma}[Lemma 3.1, \cite{Thomsen95}]
Let $A$ be a unital C*-algebra. Then
\begin{equation}
\ker \ov\Delta^n = CU_n^0(A).
\end{equation}
\end{lemma}

It is not in general true that the kernel of $\Delta^n$ can be identified with the derived group $DU_n^0(A)$, although there are several positive results \cite{dlHS84b,Thomsen93,Ng14,NgRobert17,NgRobert15}. 

It immediately follows that the quotient of $U_n^0(A)$ by the closure of the commutator subgroup (under the inductive limit topology in the $n = \infty$ case) can be identified with a quotient of $\Aff T(A)$.

\begin{theorem}[Theorem 3.2, \cite{Thomsen95}]\label{theorem:thomsen-iso}
$\ov\Delta^n$ gives a homeomorphic group iso\hyp{}morphism
\begin{equation}
U_n^0(A)/CU_n^0(A) \simeq \Aff T(A)/\ov{\tD^n(\pi_1(U_n^0(A)))}
\end{equation}
for every $n \in \bN \cup \{\infty\}$. In particular,
\begin{equation}
U_{\infty}^0(A)/CU_{\infty}^0(A) \simeq \Aff T(A)/\ov{\rho_A(K_0(A))}.
\end{equation}
\end{theorem}

\subsection{The $KT_u$-invariant}

Following \cite{CGSTW23}, for a unital C*-algebra, we let
\begin{equation}
KT_u(A) := (K_0(A),[1_A]_0,K_1(A),\rho_A,\Aff T(A))
\end{equation}
be the invariant consisting of the $K_0$-group, the position of the unit in $K_0$, the $K_1$ group, the pairing between $K_0$ and traces, and the continuous real-valued affine functions on the trace simplex (viewed as a partially ordered Banach space with order unit). 
For two unital C*-algebras $A,B$, a $KT_u$-morphism
\begin{equation}
(\alpha_0,\alpha_1,\gamma): KT_u(A) \to KT_u(B),
\end{equation}
will be a triple $(\alpha_0,\alpha_1,\gamma)$ consisting of $\alpha_0: K_0(A) \to K_0(B)$ a group homomor\hyp{}phism such that $\alpha_0([1_A]_0) = [1_B]_0$, $\alpha_1: K_1(A) \to K_1(B)$ a group homomorphism, and $\gamma: \Aff T(A) \to \Aff T(B)$ is a unital positive map such that
\begin{equation}
\begin{tikzcd}
K_0(A) \arrow[r, "\rho_A"] \arrow[d, "\alpha_0"'] & \Aff T(A) \arrow[d, "\gamma"] \\
K_0(B) \arrow[r, "\rho_B"']                       & \Aff T(B)                    
\end{tikzcd}
\end{equation}
commutes. 

We note that for large classes of unital simple C*-algebras -- for example the class of unital, simple, separable, nuclear $\cZ$-stable C*-algebras satisfying the UCT -- $KT_u(\cdot)$ recovers the Elliott invariant. 

\section{Continuous unitary group homomorphisms and traces}\label{section:ugh-ctugh-and-traces}

\hspace{\parindent}Throughout, $A$ and $B$ will be unital C*-algebras with non-empty trace simplices, and $\theta: U^0(A) \to U^0(B)$ will denote a uniformly continuous group homomorphism between the connected components of the identities in the  respective unitary groups. We will specify any additional assumptions as we go along. As $\theta$ is a continuous group homomorphism, it send commutators to commutators and limits of commutators to limits of commutators. Thus there are induced group homomorphisms
\begin{equation}
\begin{split}
&U^0(A)/CU^0(A) \to U^0(B)/CU^0(B) \text{ and }\\
&U^0(A)/DU^0(A) \to U^0(B)/DU^0(B).
\end{split}
\end{equation}
Thomsen's isomorphism \cite[Theorem 3]{Thomsen95} then brings about maps between quotients of $\Aff T(A)$ and $\Aff T(B)$:
\begin{equation}\label{eq:ch-commute}
\begin{tikzcd}
U^0(A)/CU^0(A) \arrow[r, "\simeq"] \arrow[d] & \Aff T(A)/\ov{\tD_A^1(\pi_1(U^0(A)))} \arrow[d] \\
U^0(B)/CU^0(B) \arrow[r, "\simeq"']          & \Aff T(B)/\ov{\tD_B^1(\pi_1(U^0(B)))}     .    
\end{tikzcd}
\end{equation}
In a similar vein, when $DU^0(A) = \ker\Delta_A^1$ and $DU^0(B) = \ker\Delta_B^1$, there is a purely algebraic variant of the above diagram:
\begin{equation}\label{eq:h-commute}
\begin{tikzcd}
U^0(A)/DU^0(A) \arrow[r, "\simeq"] \arrow[d] & \Aff T(A)/\tD_A^1(\pi_1(U^0(A))) \arrow[d] \\
U^0(B)/DU^0(B) \arrow[r, "\simeq"']          & \Aff T(B)/\tD_B^1(\pi_1(U^0(B))).
\end{tikzcd}
\end{equation}
In fact, there is always a diagram as above with $DU^0(A)$ and $DU^0(B)$ replaced with $\ker\Delta_A^1$ and $\ker\Delta_B^1$, respectively, by Proposition \ref{prop:pre-det-commuting-diagrams}(3). That is, whether or not the kernel of the determinant agrees with the commutator subgroup of the connected component of the identity, we have:
\begin{equation}\label{eq:h-commute-pre-det}
\begin{tikzcd}
U^0(A)/\ker\Delta_A^1 \arrow[r, "\simeq"] \arrow[d] & \Aff T(A)/\tD_A^1(\pi_1(U^0(A))) \arrow[d] \\
U^0(B)/\ker\Delta_B^1 \arrow[r, "\simeq"']          & \Aff T(B)/\tD_B^1(\pi_1(U^0(B))).
\end{tikzcd}
\end{equation}
The diagram in (\ref{eq:h-commute}) is just a special case of (\ref{eq:h-commute-pre-det}).
 
 In the setting where $\pi_1(U^0(A)) \to K_0(A)$ and $\pi_1(U^0(B)) \to K_0(B)$ are surjections, we have induced maps between quotients
\begin{equation}
\begin{split}
&\Aff T(A)/\ov{\rho_A(K_0(A))} \to \Aff T(B)/\ov{\rho_B(K_0(B))}, \\
&\Aff T(A)/\rho_A(K_0(A)) \to \Aff T(B)/\rho_B(K_0(B))
\end{split}
\end{equation}
in the respective Hausdorffized and non-Hausdorffized settings. 

One question to be answered is whether or not we can lift the maps on the right of (\ref{eq:ch-commute}) and (\ref{eq:h-commute-pre-det}) to maps $\Aff T(A) \to \Aff T(B)$. These spaces have further structure as dimension groups with order units \cite[Chapter 7]{Goodearlbook}, so we would like to be able to alter the lift to get a map which is unital and positive. We show that we can always lift this map,  and altering it to be unital and positive is possible under a certain continuity assumption on $\theta$.

If we further assume that $K_0(A) \simeq \pi_1(U^0(A))$ and $K_0(B) \simeq \pi_1(U^0(B))$ in the canonical way (which is true in the presence of $\cZ$-stability by \cite{Jiang97} or stable rank one \cite{Rieffel87}), we would like this map to be compatible with the group homomorphism
\begin{equation}
K_0(\theta): K_0(A) \to K_0(B)
\end{equation}
arising from the diagram 
\begin{equation}
\begin{tikzcd}
K_0(A) \arrow[r, "\simeq"] \arrow[d, "K_0(\theta)"'] & \pi_1(U^0(A)) \arrow[d, "\pi_1(\theta)"] \\
K_0(B) \arrow[r, "\simeq"']                           & \pi_1(U^0(B)).                          
\end{tikzcd}
\end{equation}
By compatible, we mean that 
\begin{equation}
\begin{tikzcd}
K_0(A) \arrow[d, "\pi_1(\theta)"'] \arrow[r, "\rho_A"] & \Aff T(A) \arrow[d] \\
K_0(B) \arrow[r, "\rho_B"']                          & \Aff T(B)          
\end{tikzcd}
\end{equation}
commutes, where the map on the right is the lift coming from maps induced by (\ref{eq:ch-commute}) and (\ref{eq:h-commute-pre-det}). If our map between spaces of affine continuous functions is not unital and positive, but we can alter it accordingly, we must do the same to our map between $K_0$. We would still have a commuting diagram as above, but it would give that maps induced on $K_0(\cdot)$ and $\Aff T(\cdot)$ respect the pairing. 

Stone's theorem \cite[Section X.5]{Conway19} allows one to recover from a strongly continuous one parameter family  $U(t)$ of unitaries a (possibly unbounded) self-adjoint operator $X$ such that $U(t) = e^{itX}$ for all $t \in \bR$. If it is a norm-continuous one parameter family of unitaries, one can recover a bounded self-adjoint operator $X$, and $X$ will lie in the C*-algebra generated by the unitaries. 
The use of Stone's theorem to deduce that continuous group homomorphisms between unitary groups send exponentials to exponentials is not new. Sakai used it in the 1950's in order to show that a norm-continuous group isomorphism between  unitary groups of AW*-algebras are induced by a *-isomorphism or conjugate-linear *-isomorphism between the algebras themselves \cite{Sakai55}. More recently, this sort of idea has been used to understand how the metric structure of the unitary groups can be related to the Jordan *-algebra structure of the algebras \cite{HatoriMolnar14}.

\begin{lemma}\label{lem:stone-theorem-cons}
Let $A,B$ be unital C*-algebras and $\theta: U^0(A) \to U^0(B)$ be a continuous group homomorphism. Suppose that $a \in A_{sa}$ and represent $B \subseteq \bh$ faithfully. Then $(\theta(e^{2\pi i t a}))_{t \in \bR}$ is a one-parameter norm-continuous family of unitaries, and consequently is of the form $(e^{2\pi i t b})_{t \in \bR}$ for a unique $b \in B_{sa}$. 
\begin{proof}
Using the fact that $\theta$ is a norm-continuous homomorphism, $t \mapsto \theta(e^{2\pi i t a})$ is a norm-continuous one-parameter family of unitaries. Stone's theorem gives that there is a unique self-adjoint $b \in \bh$ such that $\theta(e^{2\pi i ta}) = e^{2\pi i tb}$ for all $t \in \bR$. The boundedness of $b$ follows from norm-continuity, and the uniqueness follows from the fact that $b = \frac{1}{2\pi i t_0} \log \theta(e^{2\pi i a t_0})$ for sufficiently small $t_0 > 0$. This also gives that $b \in B$ by functional calculus. 
\end{proof}
\end{lemma}

Let $S_\theta: A_{sa} \to B_{sa}$ be defined via the correspondence given above: 
\begin{equation}
\theta(e^{2\pi i t a}) = e^{2\pi i t S_\theta(a)} \text{ for all } t \in \bR.
\end{equation}
Then $S_\theta$ is a bounded $\bR$-linear map (see \cite{Sakai55,HatoriMolnar14}, or note that it is easy to see that its kernel is closed). It is also easily checked to respect commutation, and that its canonical extension to a map from $A$ to $B$ actually sends commutators to commutators, although we will not explicitly use this. Recall that for a C*-algebra $A$, $A_0$ denotes the set of self-adjoint elements that vanish on every tracial state. 

\begin{lemma}\label{lem:sa-bijective}
If $\theta:U^0(A) \to U^0(B)$ is a continuous group homomorphism, then $S_\theta$ is a bounded linear map and the following hold.
\begin{enumerate}
\item If $\theta$ is injective, then $S_\theta$ is injective. 
\item If $\theta$ is a homeomorphism, then $S_\theta$ is bijective.
\end{enumerate}
\begin{proof}
As already remarked, $S_\theta$ is bounded. Assuming that $\theta$ is injective, suppose that $S_\theta(a) = S_\theta(b)$. Then
\begin{equation}
\theta(e^{2\pi i t a}) = \theta(e^{2\pi i t b})
\end{equation}
for all $t \in \bR$. Injectivity of $\theta$ gives that that $e^{2\pi i t a} = e^{2\pi i t b}$ for all $t \in \bR$. But this implies that $a = b$ since we can take $t$ appropriately close to 0 and take logarithms. 

Now if we further assume that $\theta$ is a homeomorphism, then $(\theta^{-1}(e^{2\pi i t b}))_{t \in \bR} \subseteq U^0(A)$ is a norm-continuous one-parameter family of unitaries which we can write as $(e^{2\pi i t a})_{t \in \bR}$ for some $a \in A_{sa}$. But then
\begin{equation}
\theta(e^{2\pi i t a}) = \theta \circ \theta^{-1}(e^{\pi i t b}) = e^{2\pi i t b}.
\end{equation}
Thus $S_\theta(a) = b$ by the uniqueness in Lemma \ref{lem:stone-theorem-cons}.
\end{proof}
\end{lemma}

Recall that a linear map $\tau: A_{sa} \to E$, where $E$ is a real Banach space, is a bounded trace is if it is a bounded $\bR$-linear map such that $\tau(a^*a) = \tau(aa^*)$ for all $a \in A$. Note that this is equivalent to $\tau \circ \Ad_u = \tau$ for all $u \in U(A)$ -- to see this, follow the steps in \cite[Exercise 3.6]{RordamKBook} together with a complexification argument. This is further equivalent to $\tau \circ \Ad_u = \tau$ for all $u \in U^0(A)$ (as one can write every element in a C*-algebra as the linear combination of 4 unitaries, each of which can be made to not have full spectrum). 

\begin{prop}\label{prop:induced-aff-map}
Let $\theta: U^0(A) \to U^0(B)$ be a continuous group homo\hyp{}morphism, $E$ be a real Banach space and $\tau: B_{sa} \to E$ a bounded trace. Then $\tau \circ S_\theta: A_{sa} \to E$ is a bounded trace. In particular, $S_\theta(A_0) \subseteq B_0$ and $S_\theta$ induces a bounded $\bR$-linear map
\begin{equation}
\Lambda_\theta: \Aff T(A) \to \Aff T(B).
\end{equation}
\begin{proof}
Observe that, for $a \in A_{sa}$ and $u \in U^0(A)$, we have
\begin{equation}
\begin{split}
e^{2\pi i t S_\theta(uau^*)} &= \theta(e^{2\pi i t uau^*}) \\
&= \theta(ue^{2\pi i t a}u^*) \\\
&= \theta(u)e^{2\pi i t S_\theta(a)}\theta(u)^* \\
&= e^{2\pi i t \theta(u)S_\theta(a)\theta(u)^*}
\end{split}
\end{equation}
for all $t \in \bR$. Therefore $S_\theta(uau^*) = \theta(u)S_\theta(a)\theta(u)^*$, and applying $\tau$ yields
\begin{equation}
\begin{split}
\tau \circ S_\theta(uau^*)  &= \tau\left(\theta(u)S_\theta(a)\theta(u)^*\right) \\
&= \tau ( S_\theta(a)),
\end{split}
\end{equation}
i.e., $\tau \circ S_\theta$ is tracial.

Thus if $a \in A_0$, it vanishes on every tracial state (hence on every tracial functional), and so $\tau \circ S_\theta(a) = 0$ for all $\tau \in T(B)$. Therefore $S_\theta(A_0) \subseteq B_0$ and so $S_\theta$ factors through a map
\begin{equation}
\Lambda_\theta: \Aff T(A) \simeq A_{sa}/A_0 \to B_{sa}/B_0 \simeq \Aff T(B),
\end{equation}
where we identify $A_{sa}/A_0 \simeq \Aff T(A)$ and $B_{sa}/B_0 \simeq \Aff T(B)$. 
\end{proof}
\end{prop}

One cannot expect $\Lambda_\theta$ (or $S_\theta$) to be unital or positive, as the following examples show. 

\begin{example}\label{example:nonpositive-example}\mbox{}
\begin{enumerate}
\item Consider a continuous homomorphism $\theta: \bT \to \bT = U^0(\bC) = U(\bC)$. By Pontryagin duality, $\theta(z) = z^n$ for some $n \in \bZ$. We then have that $S_\theta: \bR \to \bR$ is given by $S_\theta(x) = nx$. If $n \neq 1$, clearly $S_\theta$ is not unital. If $n < 0$, then $S_\theta$ is not positive since it sends $1$ to $n < 0$. An important observation, however, is that if $n < 0$, $-S_\theta: \bR \to \bR$ is positive, and $-\frac{1}{n}S$ is unital and positive.
\item Consider $\theta: \bT^3 \to \bT$ given by $\theta(z,w,v) = \ov{z}wv$. The corresponding map $S_\theta: \bR^3 \to \bR$ is given by
\begin{equation}
S_\theta(a,b,c) = -a + b + c.
\end{equation}
Clearly $(1,0,0) \in \bC^3$ is a positive element, however $S_\theta(1,0,0) = -1 < 0$. This map is however unital.
\item Let $\theta: U_2 \to \bT$ be defined by $\theta(u) = \det(u)$. Then $S_\theta: (M_2)_{sa} \to \bR$ is defined by $S_\theta(A) = \tr A$, where $\tr$ is the unnormalized trace. Clearly this map is not unital, but it is positive.
\item Let $\theta: U_2 \to U_3$ be defined by $\theta(u) = u \oplus 1$. Then $S_\theta: (M_2)_{sa} \to (M_3)_{sa}$ is given by $S_\theta(A) = A \oplus 0$, which is again not unital, but is positive. The induced map $\Lambda_\theta: \bR \to \bR$ is given by $\Lambda_\theta(x) = \frac{2}{3}x$ for $x \in \bR$.
\item Let $\theta: \bT \into U_2$ be defined by
\begin{equation}
\theta(\lambda) = \begin{pmatrix}
\lambda \\ & \lambda
\end{pmatrix}.
\end{equation}
Then $S_\theta$ is a unital, positive isometry and $\Lambda_\theta$ gives rise the identity map
\begin{equation}
\bR = \Aff T(\bC)  \to \Aff T(M_2) = \bR.
\end{equation}
\item Let $\theta: \bT \into U_2$ be defined by
\begin{equation}
\theta(\lambda) =  \begin{pmatrix}
\lambda & \\ & \ov{\lambda}
\end{pmatrix}.
\end{equation}
Then $S_\theta: \bR \to (M_2)_{sa}$ is defined by
\begin{equation}
S_\theta(x) = \begin{pmatrix}
x & \\ & -x
\end{pmatrix}
\end{equation}
and $\Lambda_\theta$ is identically zero. 
\end{enumerate}
\end{example}

The above examples are important. If $\theta(\bT) \subseteq \bT$, which is a moderate assumption (e.g., if $\theta$ was the restriction of a unital *-homomorphism or an anti-*-homomorphism), we can restrict the homomorphism to the circle to get a continuous group homomorphism $\bT \to \bT$. We understand such homomorphisms by Pontryagin duality \cite[Chapter 4]{Follandbook16}. 

We now use (pre-)determinant techniques in order to show desirable relation\hyp{}ships between our maps.

\begin{prop}\label{prop:pre-det-commuting-diagrams}
Let $A,B$ be unital C*-algebras, $\theta: U^0(A) \to U^0(B)$ be a uniformly continuous homomorphism, $E$ a real Banach space and $\tau: B_{sa} \to E$ a bounded trace.
\begin{enumerate}
\item Let $\xi: [0,1] \to U^0(A)$ be a piece-wise smooth path with $\xi(0) = 1$. Then
\begin{equation}
\tD_{\tau \circ S_\theta}^1(\xi) = \tD_\tau^1(\theta \circ \xi).
\end{equation}
In particular, $\tD_{\tau\circ S_\theta}^1(\pi_1(U^0(A))) \subseteq \tD_\tau^1(\pi_1(U^0(B)))$. 
\item The following diagram commutes:
\begin{equation}
\begin{tikzcd}
\pi_1(U^0(A)) \arrow[d, "\pi_1(\theta)"'] \arrow[r, "\tD_{\tau \circ S_\theta}^1"] & E \arrow[d, "\id"] \\
\pi_1(U^0(B)) \arrow[r, "\tD_\tau^1"']                                             & E          .
\end{tikzcd}
\end{equation}
\item The following diagram commutes:
\begin{equation}\label{eq:det-commutation}
\begin{tikzcd}
U^0(A) \arrow[r, "\Delta_{\tau \circ S_\theta}^1"] \arrow[d, "\theta"'] & E/\tD_{\tau\circ S_\theta}^1(\pi_1(U^0(A))) \arrow[d] \\
U^0(B) \arrow[r, "\Delta_\tau^1"']                                      & E/\tD_\tau^1(\pi_1(U^0(B)))           ,              
\end{tikzcd}
\end{equation}
where the map on the right is the canonical map induced from the inclusion  $\tD_{\tau \circ S_\theta}^1(\pi_1(U^0(A))) \subseteq \tD_\tau^1(\pi_1(U^0(B)))$ coming from (1).
In particular, $\theta(\ker \Delta_{\tau\circ S_\theta}^1) \subseteq \ker \Delta_\tau^1$. 

The analogous diagram commutes if we consider Thomsen's variant of the de la Harpe-Skandalis determinant associated to $\tau$ and $\tau \circ S_\theta$ in (\ref{eq:det-commutation}). 
\end{enumerate}
\begin{proof}
By the proof of Proposition \ref{rem:cts-to-piece-wise-smooth}, we can find $k \in \bN$ and $a_1,\dots,a_k \in (M_n(A))_{sa}$ such that $\xi$ is homotopic to the path $\eta: [0,1] \to U(A)$ given by
\begin{equation}\label{eq:exponential-prod}
\eta(t) = \xi\left(\frac{j-1}{k}\right)e^{2\pi i(kt - j + 1)a_j}, t \in \left[\frac{j-1}{k},\frac{j}{k}\right]
\end{equation}
and we have that
\begin{equation}\label{eq:pre-det-for-all-omega}
\tD_\omega^n(\xi) = \tD_\omega^n(\eta) = \sum_{j=1}^k \omega(a_j)
\end{equation}
whenever $\omega: A_{sa} \to F$ is a bounded trace to a real Banach space $F$. 

Now $\theta \circ \xi$ is homotopic to $\theta \circ \eta$, which has the following form, for $j=1,\dots,k$ and $t \in [\frac{j-1}{k},\frac{j}{k}]$:
\begin{equation}
\begin{split}
\theta \circ \eta(t) &= \theta \circ \xi\left(\frac{j-1}{k}\right)\theta\left(e^{2\pi i (kt - j + 1)a_j}\right) \\
&= \theta \circ \xi\left(\frac{j-1}{k}\right)e^{2\pi i (kt - j + 1)S_\theta(a_j)}.
\end{split}
\end{equation}
By taking $X_j := \theta \circ \xi\left(\frac{j-1}{k}\right)$ and $Y_j := e^{2\pi i (kt - j + 1)S_\theta(a_j)}$ in (\ref{eq:xi-to-eta-ajs}), we see that
\begin{equation}
\tD_\tau^1(\theta \circ \xi) = \tD_\tau^1(\theta \circ \eta) = \sum_{j=1}^k \tau(S_\theta(a_j)),
\end{equation}
and this last quantity is just $\tD_{\tau \circ S_\theta}^1(\xi)$ by (\ref{eq:pre-det-for-all-omega}). 

As for parts (2) and (3), they follow from (1). Indeed, (2) is obvious and to see (3) we let
\begin{equation}
q: E/\tD_{\tau \circ S_\theta}^1(\pi(U^0(A))) \to E/\tD_\tau^1(\pi(U^0(B)))
\end{equation}
be the canonical surjection. If $\xi_u: [0,1] \to U^0(A)$ is a path from 1 to $u$, then we have that
\begin{equation}
\begin{split}
q(\Delta_{\tau \circ S_\theta}^1(u)) &= q\left(\tD_{\tau \circ S_\theta}^1(\xi_u) + \tD_{\tau \circ S_\theta}^1(\pi_1(U^0(A)))\right) \\
&= \tD_{\tau \circ S_\theta}^1(\xi_u) + \tD_\tau^1(\pi_1(U^0(B))) \\
&\overset{(1)}{=} \tD_\tau^1(\theta \circ \xi_u) + \tD_\tau^1(\pi_1(U^0(B))) \\
&= \Delta_\tau^1(\theta(u)).
\end{split}
\end{equation}

The remark about Thomsen's variant is obvious. 
\end{proof}
\end{prop}

\begin{cor}\label{cor:pairing}
The following diagram commutes:
\begin{equation}\label{eq:cor:pairing-diagram1}
\begin{tikzcd}
\pi_1(U^0(A)) \arrow[d, "\pi_1(\theta)"'] \arrow[r, "\tD_A^1"] & \Aff T(A) \arrow[d, "\Lambda_\theta"] \\
\pi_1(U^0(B)) \arrow[r, "\tD_B^1"']                            & \Aff T(B).                            
\end{tikzcd} 
\end{equation}
In particular, when the canonical maps
\begin{equation}
\pi_1(U^0(A)) \to K_0(A) \text{ and }\pi_1(U^0(B)) \to K_0(B)
\end{equation}
 are isomorphisms, we have that
\begin{equation}
\begin{tikzcd}
K_0(A) \arrow[d, "K_0(\theta)"'] \arrow[r, "\rho_A"] & \Aff T(A) \arrow[d, "\Lambda_\theta"] \\
K_0(B) \arrow[r, "\rho_B"']                            & \Aff T(B)                              
\end{tikzcd} 
\end{equation}
commutes, where $K_0(\theta): K_0(A) \to K_0(B)$ is the map induced from the diagram
\begin{equation}
\begin{tikzcd}
K_0(A) \arrow[r, "\simeq"] \arrow[d, "K_0(\theta)"'] & \pi_1(U^0(A)) \arrow[d, "\pi_1(\theta)"] \\
K_0(B) \arrow[r, "\simeq"']                          & \pi_1(U^0(B)).                         
\end{tikzcd}
\end{equation}
\begin{proof}
By definition, we have that $\Lambda_\theta \circ \Tr_A = S_\theta \circ \Tr_B$ and so if $\xi: [0,1] \to U^0(A)$ is a (piece-wise smooth) path, we have
\footnotetext[3]{If $\fX,\fY$ are Banach spaces over the same base field, $f: [0,1] \to \fX$ is a Riemann integrable function and $T: \fX \to \fY$ is a bounded linear map, then $T \circ f: [0,1] \to \fY$ is Riemann integrable with $T\left(\int_0^1 f(t)dt\right) = \int_0^1 T \circ f(t) dt$.}
\begin{equation}\label{eq:lambda-S-commutes-with-Tr}
\begin{split}
\Lambda_\theta\left(\tD_{\Tr_A}^1(\xi)\right) &= \Lambda_\theta\left(\int_0^1 \Tr_A\left(\xi'(t)\xi(t)^{-1}\right)dt\right) \\
&= \int_0^1 \Lambda_\theta \circ \Tr_A\left(\xi'(t)\xi(t)^{-1}\right)dt\ \footnotemark \\
&= \int_0^1 \Tr_B \circ S_\theta\left(\xi'(t)\xi(t)^{-1}\right)dt \\
&= \tD_{\Tr_B \circ S_\theta}(\xi).
\end{split}
\end{equation}
Now we note that the diagram
\begin{equation}
\begin{tikzcd}
                                                                                                           & \Aff T(A) \arrow[rd, "\Lambda_\theta"] &                            \\
\pi_1(U^0(A)) \arrow[ru, "\tD_A^1"] \arrow[d, "\pi_1(\theta)"'] \arrow[rr, "\tD_{\Tr_B \circ S_\theta}^1"] &                                        & \Aff T(B) \arrow[d, "\id"] \\
\pi_1(U^0(B)) \arrow[rr, "\tD_B^1"']                                                                       &                                        & \Aff T(B)                 
\end{tikzcd}
\end{equation}
commutes, where the top triangle commutes by the (\ref{eq:lambda-S-commutes-with-Tr}) and the bottom square commutes by Proposition \ref{prop:pre-det-commuting-diagrams}(2). This gives that (\ref{eq:cor:pairing-diagram1}) commutes.

The second part follows since if $\xi$ is a (piece-wise smooth) path in $U_n(A)$ and $m > n$, then
\begin{equation}
\tD_\tau^m(\xi \oplus 1_{m-n}) = \tD_\tau^n(\xi).
\end{equation}
\end{proof}
\end{cor}

\begin{prop}\label{prop:S-theta-is-a-lift}
Let $A,B$ be unital C*-algebras. Then the map $\Lambda_\theta$ is a lift of the map
\begin{equation}
\Aff T(A)/\tD_A^1(\pi_1(U^0(A))) \to \Aff T(B)/\tD_B^1(\pi_1(U^0(B)))
\end{equation}
as described in (\ref{eq:h-commute-pre-det}).
\begin{proof}
Let us label the maps in the diagram (\ref{eq:h-commute-pre-det}):
\begin{equation}
\begin{tikzcd}
U^0(A)/\ker \Delta^1_A \arrow[d, "\tilde\theta"'] \arrow[r, "\delta_A^1"] & \Aff T(A)/\tD_A^1(\pi_1(U^0(A))) \arrow[d, "P"] \\
U^0(B)/\ker \Delta^1_B \arrow[r, "\delta_B^1"']                      & \Aff T(B)/\tD_B^1(\pi_1(U^0(B)))                           
\end{tikzcd}
\end{equation}
where $\tilde\theta([u]) := [\theta(u)]$, $\delta_A^1([e^{2\pi i a}]) := [\hat{a}]$, $\delta_B^1$ is defined similarly, and
\begin{equation}
P = \delta_B^1 \circ \tilde\theta \circ (\delta_A^1)^{-1}.
\end{equation}
But then we have that
\begin{equation}
\begin{split}
P([a]) &= \delta_B^1 \circ \tilde\theta \circ (\delta_A^1)^{-1}([\hat{a}]) \\
&= \delta_B^1 \circ \tilde\theta([e^{2\pi i a}]) \\
&= \delta_B^1([\theta(e^{2\pi i a})]) \\
&= \delta_B^1([e^{2\pi i S_\theta(a)}]) \\
&= [\widehat{S_\theta(a)}] \\
&= [\Lambda_\theta(\widehat{a})]
\end{split}
\end{equation}
In particular, the diagram
\begin{equation}
\begin{tikzcd}
\Aff T(A) \arrow[d, "\Lambda_\theta"'] \arrow[r, "q_A^1"] & \Aff T(A)/\tD_A^1(\pi_1(U^0(A))) \arrow[d, "P"] \\
\Aff T(B) \arrow[r, "q_B^1"']                         & \Aff T(B)/\tD_B^1(\pi_1(U^0(B)))               
\end{tikzcd}
\end{equation}
commutes, where $q_A^1$ and $q_B^1$ are the respective quotient maps. 
\end{proof}
\end{prop}

\begin{prop}\label{prop:S-theta-is-a-lift-H}
Let $A,B$ be unital C*-algebras and $\theta: U^0(A) \to U^0(B)$ be a continuous group homomorphism. Then $\Lambda_\theta: \Aff T(A) \to \Aff T(B)$ is a lift of the map
\begin{equation}
\Aff T(A)/\ov{\tD_A^1(\pi_1(U^0(A)))} \to \Aff T(B)/\ov{\tD_B^1(\pi_1(U^0(B)))}
\end{equation}
as described in (\ref{eq:ch-commute}).
\begin{proof}
One can mimic the proof above or apply the above result and appeal to the commuting diagram 
\begin{equation}
\begin{tikzcd}
 \Aff T(A)/\tD_A^1(\pi_1(U^0(A))) \arrow[d] \arrow[r] &  \Aff T(A)/\ov{\tD_A^1(\pi_1(U^0(A)))} \arrow[d] \\
\Aff T(B)/\tD_B^1(\pi_1(U^0(B)))    \arrow[r]         & \Aff T(B)/\ov{\tD_B^1(\pi_1(U^0(B)))},           
\end{tikzcd}
\end{equation}
where the vertical maps are defined via the diagrams from (\ref{eq:h-commute-pre-det}) and (\ref{eq:ch-commute}) respectively, and the horizontal maps are the canonical surjections. 
\end{proof}
\end{prop}

In particular, assuming some $K_0$-regularity gives that $\Lambda_\theta$ is a lift of a map between quotients of spaces of continuous real-valued affine functions by images of $K_0$.

\begin{cor}
Let $A,B$ be unital C*-algebras such that the canonical maps
\begin{equation}
\pi_1(U^0(A)) \to K_0(A) \text{ and }\pi_1(U^0(B)) \to K_0(B)
\end{equation}
are surjections. If $\theta: U^0(A) \to U^0(B)$ is a continuous homomorphism, then $\Lambda_\theta$ is a lift of the maps on the right of the following two commutative diagrams:
\begin{equation}
\begin{tikzcd}
U^0(A)/CU^0(A) \arrow[r, "\simeq"] \arrow[d] & \Aff T(A)/\ov{\rho_A(K_0(A))} \arrow[d] \\
U^0(B)/CU^0(B) \arrow[r, "\simeq"']          & \Aff T(B)/\ov{\rho_B(K_0(B))}  
\end{tikzcd}
\end{equation}
and
\begin{equation}
\begin{tikzcd}
U^0(A)/\ker \Delta_A^1 \arrow[r, "\simeq"] \arrow[d] & \Aff T(A)/\rho_A(K_0(A)) \arrow[d] \\
U^0(B)/\ker \Delta_B^1\arrow[r, "\simeq"']          & \Aff T(B)/\rho_B(K_0(B))    .    
\end{tikzcd}
\end{equation}
Further, if $\ker \Delta_A^1 = DU^0(A)$ and $\ker \Delta_B^1 = DU^0(B)$, then $\Lambda_\theta$ is a lift of the map induced by the diagram
\begin{equation}
\begin{tikzcd}
U^0(A)/DU^0(A) \arrow[r, "\simeq"] \arrow[d] & \Aff T(A)/\rho_A(K_0(A)) \arrow[d] \\
U^0(B)/DU^0(B)\arrow[r, "\simeq"']          & \Aff T(B)/\rho_B(K_0(B))    .    
\end{tikzcd}
\end{equation}
\end{cor}

C*-algebras satisfying the last condition arise naturally -- for example unital, separable, simple, pure C*-algebras of stable rank 1 such that every 2-quasi\hyp{}tracial state on $A$ is a trace have this property \cite{NgRobert17}.

\section{The order on $\Aff T(\cdot)$}\label{section:ugh-order-on-aff-classification}

\hspace{\parindent}We now examine when the map induced on $\Aff T(\cdot)$ is positive in order to compare $K$-theory, traces, and the pairing. As we saw in Example \ref{example:nonpositive-example}, the map we get between spaces of affine functions on the trace simplices need not be positive nor unital in general. In this section, we will be able to use the map $\Lambda_\theta$ to construct a unital positive map, under some extra assumptions on $\theta$.

We record the following results as they give us necessary and sufficient conditions for the $\Lambda_\theta$ to be positive. We use the C*-algebra-valued analogue of the fact that any unital, contractive linear functional is positive, along with the fact that completely positive maps are (completely) bounded with the norm determined by the image of the unit. Recall that an operator system is a self-adjoint unital subspace of a C*-algebra. The following is a combination of Proposition 2.11, Theorem 3.9 and Proposition 3.6 in \cite{Paulsenbook}.

\begin{prop}\label{prop:abelian-operator-systems}
Let $\cS$ be an operator system and $B$ a unital C*-algebra.
\begin{enumerate}
\item If $\phi: \cS \to B$ is a unital contraction, then $\phi$ is positive. \label{unitalcontraction}
\item If $B = C(X)$ and $\phi: \cS \to B$ is positive, then it is bounded with $\|\phi\| = \|\phi(1)\|$. \label{abeliantarget}
\end{enumerate}
\end{prop}

\begin{lemma}\label{lem:inj-st-pm1}
Suppose $A,B$ are unital C*-algebras and $\theta: U^0(A) \to U^0(B)$ be a continuous group homomorphism such that $\theta(\bT) = \bT$. If $\theta|_\bT$ is injective, then $S_\theta(1) \in \{1,-1\}$. 
\begin{proof}
The restriction $\theta|_{\bT}: \bT \to \bT$ is a continuous group homomorphism, hence by Pontryagin duality is of the form $\theta(z) = z^n$ for some $n$. Injectivity implies that $n \in \{1,-1\}$. We then have that
\begin{equation}
e^{2\pi i S_\theta(1)t} = \theta(e^{2\pi i t}) = e^{2\pi i nt}
\end{equation}
for all $t \in \bR$. This implies that $S_\theta(1) = n\cdot 1 \in \{1,-1\}$. 
\end{proof}
\end{lemma}


In the sequel, we will be interested in the case that $T(A) \neq \emptyset$. The map $S_\theta$ will still descend to a map $\Lambda_\theta: A_{sa}/A_0 \to B_{sa}/B_0$ regardless of their being tracial states on either C*-algebra, but these quotients are zero if there are no traces. If we take the zero Banach space to be partially ordered with order unit trivially, then one can speak of unital and positive maps $\Aff T(A) \to \Aff T(B)$, regardless of whether $T(B)$ is empty or not (any map from a partially ordered Banach space with order unit to the zero Banach space will be unital and positive). However, if $T(A) = \emptyset$ and $T(B) \neq \emptyset$, then our map $\Lambda_\theta$ would have no chance of being unital. We identify the zero Banach space with $\Aff (\emptyset) = C_\bR(\emptyset)$ and with the complex analogues as well. However, the following lemma is true regardless of whether $T(A)$ is empty or not. 

\begin{lemma}\label{lem:tildeS-isometric}
Suppose $A,B$ are unital C*-algebras and $\theta: U^0(A) \to U^0(B)$ is a continuous group homomorphism such that $\theta(\bT) = \bT$. If $\theta$ is injective, the following are equivalent:
\begin{enumerate}
\item one of $\Lambda_\theta$ or $-\Lambda_\theta$ is positive;
\item $\Lambda_\theta$ is contractive.
\end{enumerate}
\begin{proof}
First suppose that $T(A),T(B) \neq \emptyset$. By Lemma \ref{lem:inj-st-pm1}, we know that $S_\theta(1) \in \{1,-1\}$ and consequently $\Lambda_\theta(\hat 1) \in \{\hat{1},\widehat{-1}\}$ (where we recall that, for $a \in A_{sa}$, $\hat{a} \in \Aff T(A)$ is the affine function $\hat{a}(\tau) = \tau(a)$). By replacing $\Lambda_\theta$ with $-\Lambda_\theta$, we can without loss of generality assume that $\Lambda_\theta$ is unital. Using the fact that $\Aff T(A) + i\Aff T(A) \subseteq C(T(A))$ is an operator system and the canonical extension
\begin{equation}
\Lambda_\theta^\bC: \Aff T(A) + i\Aff T(A) \to \Aff T(B) + i\Aff T(B) \subseteq C(T(B))
\end{equation}
is a unital linear map with abelian target algebra, this is an easy consequence of the two parts of Proposition \ref{prop:abelian-operator-systems}.

Finally, if $T(A) = \emptyset$, then any map from $\Aff T(A)$ is both contractive and positive trivially. The same is true of any map with codomain $\Aff T(B)$ if $T(B) = \emptyset$. 
\end{proof}
\end{lemma}

\begin{lemma}\label{lem:lipschitzSbounded}
Suppose $A,B$ are unital C*-algebras and $\theta: U^0(A) \to U^0(B)$ is a continuous group homomorphism. Suppose that $T(A) \neq \emptyset$.
\begin{enumerate}
\item $\|\Lambda_\theta\| \leq \|S_\theta\|$.
\item If $K > 0$ is such that $\|\theta(u) - \theta(v)\| \leq K\|u - v\|$ for all $u,v \in U^0(A)$, then $\|S_\theta\| \leq K$ and $\|\Lambda_\theta\| \leq K$.
\item If $\theta$ is a homeomorphism, then $\Lambda_\theta$ is bijective with bounded inverse being $\Lambda_{\theta^{-1}}$.
\item  If $\theta$ is isometric, then so is $S_\theta$.
If $\theta$ is a surjective isometry, then $\Lambda_\theta$ is a surjective isometry.
\end{enumerate}

\begin{proof}
As $S_\theta(A_0) \subseteq B_0$, we have
\begin{equation}
\begin{split}
\|\Lambda_\theta(\hat{a})\|_{\Aff T(B)} &= \|\widehat{S_\theta(a)}\|_{\Aff T(B)} \\
&= \inf_{b \in B_0} \|S_\theta(a) + b\| \\
&\leq \inf_{b \in \theta(A_0)} \|S_\theta(a) + b\| \\
&= \inf_{a' \in A_0} \|S_\theta(a) + S_\theta(a')\| \\
&\leq \|S_\theta\|\inf_{a' \in A_0}\|a + a'\| \\
&= \|S_\theta\|\|\hat{a}\|_{\Aff T(A)}
\end{split}
\end{equation}
whenever $a \in A_{sa}$. This gives that $\|\Lambda_\theta\| \leq \|S_\theta\|$.

The fact that $S_\theta$ is an isometry whenever $\theta$ is an isometry follows from an argument in \cite{HatoriMolnar14};  we exemplify said argument to show the bound condition. We use the observation that
\begin{equation}
\frac{e^{2\pi i t a} - 1}{t} \to 2\pi i a
\end{equation}
as $t \to 0$. Since
\begin{equation}
\|e^{2\pi i t S_\theta(a)} - 1\| \leq K\|e^{2\pi i t a} - 1\| 
\end{equation}
for all $t \in \bR$, we can divide both sides by $\frac{1}{2\pi}|t|$ and take $t \to 0$ to get that
\begin{equation}
\|S_\theta(a)\| \leq K\|a\|.
\end{equation}
Thus $\|S_\theta\| \leq K$. It then follows from (1) that $\|\Lambda_\theta\| \leq K$ as well. 

If $\theta$ is a homeomorphism, $S_\theta$ and $S_{\theta^{-1}}$ are both defined and its clear that $S_\theta^{-1} = S_{\theta^{-1}}$.

The surjectivity of $\theta$ implies the surjectivity of $S_\theta$ and thus if $b \in B_{sa}$, we can find $a \in A_{sa}$ with $S_\theta(a) = b$. But then $\Lambda_\theta(\hat{a}) = \widehat{S_\theta(a)} = \hat{b}$. In particular, $\Lambda_\theta$ is surjective. Now if $\theta$ is a surjective isometry, we identify $\Aff T(A) \simeq A_{sa}/A_0$ and $\Aff T(B) \simeq B_{sa}/B_0$, noting that $S_\theta(A_0) = B_0$, and that $\Lambda_\theta$ will preserve the quotient norms. 
\end{proof}
\end{lemma}

\begin{cor}
Suppose $A,B$ are unital C*-algebras and $\theta: U^0(A) \to U^0(B)$ is a continuous group homomorphism. Suppose that $T(A) \neq \emptyset$. If $S_\theta(1) = n$ and $\|S_\theta\| = |n|$, then $\frac{1}{n}S_\theta$ is a unital contraction, hence positive. In particular, if $\theta(\bT) = \bT$ and $\theta|_{\bT}$ is an injection, then either $\Lambda_\theta$ or $-\Lambda_\theta$ is unital and positive. 
\begin{proof}
The first part follows from the above lemma. If $\theta$ is an injection with $\theta(\bT) = \bT$, we have that $S_\theta(\hat{1}) \in \{\hat{1},\widehat{-1}\}$ and that $\Lambda_\theta$ is contractive, so one of $\Lambda_\theta$ or $-\Lambda_\theta$ is a unital contraction, hence positive by part (1) of Proposition \ref{prop:abelian-operator-systems}. 
\end{proof}
\end{cor}

\begin{theorem}\label{thm:affine-map-from-injection}
Suppose $A,B$ are unital C*-algebras and $\theta: U^0(A) \to U^0(B)$ is a contractive injective homomorphism such that $\theta(\bT) = \bT$. Suppose that $T(A) \neq \emptyset$. Then there is a continuous affine map $T_\theta: T(B) \to T(A)$ such that $\Lambda_\theta(f) = f \circ T_\theta$ or $-\Lambda_\theta(f) = f \circ T_\theta$, depending on whether $\Lambda_\theta$ or $-\Lambda_\theta$ is positive.
\begin{proof}
This follows from the fact that the induced map $\Lambda_\theta: \Aff T(A) \to \Aff T(B)$ will have the property that $\Lambda_\theta$ or $-\Lambda_\theta$ will be a  unital positive map. Therefore by contravariant identification of compact convex sets (of locally convex Hausdorff linear spaces) with the state space of the space of continuous real-valued affine valued functions on them (\cite[Chapter 7]{Goodearlbook}), there exists a continuous affine map $T_\theta: T(B) \to T(A)$. 
\end{proof}
\end{theorem}

\begin{theorem}\label{thm:trace-result}
Suppose $A,B$ be unital C*-algebras and $\theta: U^0(A) \to U^0(B)$ is a contractive topological group isomorphism such that $\theta(\bT) = \bT$. Suppose that $T(A) \neq \emptyset$. Then the map $T_\theta: T(B) \to T(A)$ induced by $\Lambda_\theta$ is an affine homeomorphism.
\begin{proof}
As $\theta(\bT) = \bT$, $S_\theta(1) \in \{-1,1\}$. Let $\pm \Lambda_\theta: \Aff T(A) \to \Aff T(B)$ be either $\Lambda_\theta$ or $-\Lambda_\theta$, depending on which is unital, positive, contractive and surjective by combining Lemmas \ref{lem:lipschitzSbounded}, \ref{lem:tildeS-isometric} and \ref{lem:sa-bijective}(2). By the duality of (compact) simplices and continuous affine functions on them, the map $T_\theta: T(B) \to T(A)$ is an affine homeomorphism.
\end{proof}
\end{theorem}

\begin{theorem}\label{thm:KTu-morphism-from-injection}
Let $A,B$ be unital C*-algebras and $\theta: U(A) \to U(B)$ be a contractive injective homomorphism such that $\theta(\bT) = \bT$. Suppose that $T(A) \neq \emptyset$.  If 
\begin{equation}\label{eq:k-regular-KTu-morph}
\pi_i(U(A)) \simeq K_{i-1}(A)\text{ and } \pi_i(U(B)) \simeq K_{i-1}(B) \text{ for } i=0,1,\footnote{Note that $\pi_1(U(A)) = \pi_1(U^0(A))$ and $\pi_1(U(B)) = \pi_1(U^0(B))$ since we are taking our base point to be the identity in each case.}
\end{equation}
via the canonical maps, then there is an induced map
\begin{equation}
KT_u(\theta): KT_u(A) \to \KT_u(B).
\end{equation}
\begin{proof}
Let
\begin{itemize}
\item $\Lambda:= \Lambda_{\theta|_{U^0(A)}}: \Aff T(A) \to \Aff T(B)$,
\item $\theta_0: \pi_1(U^0(A)) \to \pi_1(U^0(B))$ be the map induced on fundamental groups by $\theta|_{U^0(A)}$,
\item $K_0(\theta)$ be the map induced on $K_0$ by $\theta_0$ together with (\ref{eq:k-regular-KTu-morph}) for $i = 1$,
\item $\theta_1: \pi_0(U(A)) \to \pi_0(U(B))$ be the map induced by $\theta$ on connected components (so that $\theta_1([u]_{\sim_h}) = [\theta_1(u)]_{\sim_h}$) and
\item  $K_1(\theta)$ be the map induced by $\theta_1$ together with (\ref{eq:k-regular-KTu-morph}) for $i=0$. 
\end{itemize}
Then
\begin{equation}
(\pm K_0(\theta),K_1(\theta),\pm \Lambda): KT_u(A) \to KT_u(B)
\end{equation}
is a $KT_u$-morphism, where $\pm \Lambda$ is either $\Lambda$ or $-\Lambda$ depending on which one is unital and positive, and $\pm K_0(\theta)$ is either $K_0(\theta)$ if $\Lambda$ is positive or $-K_0(\theta)$ if $-\Lambda$ is positive.  Indeed, $\pm K_0(\theta),\theta_1,\pm \Lambda$ are all appropriate morphisms, and we have that
\begin{equation}
\begin{tikzcd}
K_0(A) \arrow[d, "\pm K_0(\theta)"'] \arrow[r, "\rho_A"] & \Aff T(A) \arrow[d, "\pm \Lambda"] \\
K_0(B) \arrow[r, "\rho_B"']                            & \Aff T(B)                              
\end{tikzcd} 
\end{equation}
commutes\footnote{The map $-K_0(\theta)$ will take a piece-wise smooth loop $\xi$ to the loop $-\theta \circ \xi$ defined by $(-\theta \circ \xi)(t) = \theta(\xi(-t))$. From here its obvious that the diagram commutes.} by Corollary \ref{cor:pairing}.
\end{proof}
\end{theorem}

\begin{cor}\label{cor:KTu-isomorphism}
Let $A,B$ be unital C*-algebras, $\theta: U^0(A) \to U^0(B)$ is a contractive topological group isomorphism such that $\theta(\bT) = \bT$, and suppose that $T(A) \neq \emptyset$.  If 
\begin{equation}
\pi_i(U(A)) \simeq K_{i-1}(A)\text{ and } \pi_i(U(B)) \simeq K_{i-1}(B) \text{ for } i=0,1,
\end{equation}
via the canonical maps, then $KT_u(A) \simeq KT_u(B)$.
\begin{proof}
By Corollary \ref{thm:KTu-morphism-from-injection}, we have an induced $KT_u$-morphism. This map is necessarily an isomorphism since $\theta$ is. 
\end{proof}
\end{cor}

\begin{cor}\label{cor:reg-KT_u-morph}
Let $A,B$ be unital C*-algebras which are either $\cZ$-stable or of stable rank one and suppose that $T(A) \neq \emptyset$. Let $\theta: U(A) \to U(B)$ be a contractive injective homomorphism such that $\theta(\bT) = \bT$. Then there is an induced map
\[ KT_u(\theta): KT_u(A) \to KT_u(B). \]
\begin{proof}
C*-algebras which are $\cZ$-stable or have stable rank one satisfy the hypotheses of Theorem \ref{thm:KTu-morphism-from-injection} by \cite{Jiang97} and \cite{Rieffel87} respectively. So Theorem \ref{thm:KTu-morphism-from-injection} applies.
\end{proof}
\end{cor}

\begin{remark}\label{rem:strict-order}
The strict ordering on $\Aff T(A)$ is given by $f \gg g$ if $f(\tau) > g(\tau)$ for all $\tau \in T(A)$. If $A,B$ are unital with $T(A) \neq \emptyset$ and $\theta:U^0(A) \to U^0(B)$ is a contractive injective homomorphism such that $\theta(\bT) = \bT$, then $\pm \Lambda_\theta: \Aff T(A) \to \Aff T(B)$ is a unital positive contraction by Lemma \ref{lem:tildeS-isometric} (again $\pm \Lambda_\theta$ is $\Lambda_\theta$ or $-\Lambda_\theta$ depending on which is positive). We moreover have that
\begin{equation}
\pm \Lambda_\theta(f) \gg \pm \Lambda_\theta(g) \iff f \gg g.
\end{equation}
Indeed, let us show that $f \gg 0$ if and only if its image is $\gg 0$. As $\pm \Lambda_\theta$ has the form $\pm \Lambda_\theta(\hat{a}) = \widehat{\pm S_\theta(a)}$, it suffices to show that if $\sigma(a) > 0$ for all $\sigma \in T(A)$, then $\tau(\pm S_\theta (a)) > 0$ for all $\tau \in T(B)$.
But this is trivial because $\tau \circ \pm S_{\theta}: A_{sa} \to \bR$ extends canonically to a tracial state $A \to \bC$ by Proposition \ref{prop:induced-aff-map}, so evaluating it against $a$ must give that it is strictly positive. 
\end{remark}

The above says the following: for certain C*-algebras, we can read off positivity in $K_0$, thinking of it as the fundamental group of the unitary group, from the strict positivity of the pre-determinant applied a the loop. Precisely, a non-zero element $x \in K_0(A)$, where $A$ is a unital, simple C*-algebra with strict comparison, is in the positive cone if and only if the corresponding loop $\xi_x$ satisfies $\tD_\tau(\xi_x) > 0$ for all $\tau \in T(A)$.

Although the following is known, for example by very strong results in \cite[Chapter 6]{AraMathieubook} pertaining to certain prime C*-algebras, we give the follow\hyp{}ing as a corollary by using $K$-theoretic classification results. 

\begin{cor}
Let $A,B$ be unital, separable, simple, nuclear $\cZ$-stable C*-algebras satisfying the UCT. Then $A \simeq B$ if and only if there is a contractive isomorphism $U(A) \simeq U(B)$.
\begin{proof}
Its clear that two isomorphic C*-algebras have isomorphic unitary groups. On the other hand, if $U(A) \simeq U(B)$, then since these C*-algebras are $\cZ$-stable, Corollary \ref{cor:KTu-isomorphism} applies. As $KT_u(\cdot)$ recovers the Elliott invariant, which is a complete invariant for the C*-algebras as in the statement of the theorem (by \cite[Corollary D]{CETWW21}, \cite{EGLN15,GongLinNiu20I,GongLinNiu20II} and the references therein), $A \simeq B$.
\end{proof}
\end{cor}

Using the state of the art classification of embeddings \cite{CGSTW23}, there is an enlarged invariant of $KT_u(\cdot)$ which is able to classify morphisms between certain C*-algebras. Any $KT_u$-morphisms automatically has a lift to this larger invariant \cite[Theorem 3.9]{CGSTW23}, and so under the assumption that the $KT_u$-morphism is faithful (i.e., the map $T(B) \to T(A)$ induced by the map $\Aff T(A) \to \Aff T(B)$ sends traces on $B$ to faithful traces on $A$), there is a *-homomorphism witnessing the $KT_u$-morphism. Therefore as a corollary of their main theorem, we have that for an abundance of C*-algebras, there is an (contractive) embedding of unitary groups if and only if there is an embedding of C*-algebras.

\begin{cor}\label{cor:ugh-subgroup-subalgebra}
Let $A$ be a unital, separable, simple nuclear C*-algebra sat\hyp{}isfying the UCT which is either $\cZ$-stable or of stable rank one, and $B$ be a unital, separable, simple, nuclear $\cZ$-stable C*-algebra. If there is a contractive injective homomorphism $\theta:U(A) \to U(B)$ such that $\theta(\bT) = \bT$, then there is an embedding $A \into B$. 
\begin{proof}
Assuming such a $\theta$ exists, it gives rise to a $KT_u$-morphism
\begin{equation}
KT_u(\theta): KT_u(A) \to KT_u(B)
\end{equation}
by Corollary \ref{cor:reg-KT_u-morph}. As $A,B$ are simple, the map $T_\theta: T(B) \to T(A)$ necessarily maps traces on $B$ to faithful traces on $A$, and thus the $KT_u$-morphism $KT_u(\theta)$ is ``faithful''. Therefore $KT_u(\theta)$ induces an embedding $A \into B$ by \cite[Theorem B]{CGSTW23}. 
\end{proof}
\end{cor}

\section{A slight general linear variant}\label{section:ugh-general-linear-variants}

Here we briefly describe a slight general linear variant of some of the results above.
Unfortunately, the maps we get at the level of $A,B$ and complex-valued continuous affine functions are not $\bC$-linear in general (see Example \ref{example:non-C-linear}). 
In the presence of a continuous homomorphism $\theta: GL^0(A) \to GL^0(B)$, we have corresponding maps
\begin{equation}\label{eq:gch-commute}
\begin{tikzcd}
GL^0(A)/CGL^0(A) \arrow[r, "\simeq"] \arrow[d] & \left(A/\ov{[A,A]}\right)/\ov{\tD_A^1(\pi_1(U^0(A)))} \arrow[d] \\
GL^0(B)/CGL^0(B) \arrow[r, "\simeq"']          & \left(A/\ov{[A,A]}\right)/\ov{\tD_B^1(\pi_1(U^0(B)))}     .    
\end{tikzcd}
\end{equation}
Again, by modding out by algebraic commutator subgroups when $DGL^0(A) = \ker \Delta_A^1$ and $DGL^0(B) = \ker \Delta_B^1$ (both with respect to the general linear variant of the de la Harpe-Skandalis determinant, as originally introduced in \cite{dlHS84a}), instead of closures of derived groups, there is a purely algebraic variant of the above diagram:
\begin{equation}\label{eq:gh-commute}
\begin{tikzcd}
GL^0(A)/DGL^0(A) \arrow[r, "\simeq"] \arrow[d] &\left(A/\ov{[A,A]}\right)/\tD_A^1(\pi_1(GL^0(A))) \arrow[d] \\
GL^0(B)/DGL^0(B) \arrow[r, "\simeq"']          & \left(A/\ov{[B,B]}\right)/\tD_B^1(\pi_1(GL^0(B))).
\end{tikzcd}
\end{equation}

Thinking of $K_0(A)$ as the Grothendieck group of the semigroup of equivalence classes of idempotents and $K_0(A) \simeq \pi_1(GL_{\infty}^0(A))$, we would like to lift the maps on the right of (\ref{eq:gch-commute}) and (\ref{eq:gh-commute})
to a map
\begin{equation}
A/\ov{[A,A]} \to B/\ov{[B,B]}
\end{equation}
(the latter holding true when $A,B$ are C*-algebras whose determinant has appropriate kernel).

We can always achieve a bounded $\bR$-linear map. 

\begin{prop}\label{prop:gl-G-theta-lift}
Let $A,B$ be unital C*-algebras and $\theta: GL^0(A) \to GL^0(B)$ be a continuous group homomorphism. Then there is a continuous $\bR$-linear map
\begin{equation}
\tilde{G}_\theta: A/\ov{[A,A]} \to B/\ov{[B,B]}
\end{equation}
which lifts the maps on the right of  (\ref{eq:gch-commute}) and (\ref{eq:gh-commute}) (the latter holding when $DGL^0(A) = \ker\Delta_A^1$ and $DGL^0(B) = \ker\Delta_B^1$).
\begin{proof}

We define $G_\theta: A \to B$ given by
\begin{equation}
G_\theta(a) := \lim_n \frac{n}{2\pi i}\log \theta(e^{2\pi i \frac{a}{n}}).
\end{equation}
We note that the sequence on the right is eventually constant: choose $N$ large enough such that $n \geq N$ implies that 
\begin{equation}
\|\theta(e^{2\pi i \frac{a}{n}}) - 1\| < 1.
\end{equation}
We then have for $n \geq N$,
\begin{equation}
\begin{split}
\frac{n}{2\pi i}\log \theta(e^{2\pi i \frac{a}{n}}) &= \frac{n}{2\pi i}\log \theta(e^{2\pi i \frac{Na}{Nn}}) \\
&= \frac{n}{2\pi i}\log \theta(e^{2\pi i \frac{a}{Nn}})^N \\
&= \frac{Nn}{2\pi i}\log \theta(e^{2\pi i \frac{a}{Nn}}) \\
&= \frac{N}{2\pi i}\log \theta(e^{2\pi i \frac{a}{Nn}})^n \\
&= \frac{N}{2\pi i}\log \theta(e^{2\pi i \frac{na}{Nn}}) \\
&= \frac{N}{2\pi i}\log \theta(e^{2\pi i \frac{a}{N}}).
\end{split}
\end{equation}
To see that the map is additive, one can use the Lie product formula
\begin{equation}
e^{a+b} = \lim_k \left(e^{\frac{a}{k}}e^{\frac{b}{k}}\right)^k
\end{equation}
 (see for example \cite[Theorem 2.11]{HallBook}), along with the fact that $\theta$ is a continuous homomorphism. 
From here, it is clear that $G_\theta$ is continuous and $\bQ$-linear, hence $\bR$-linear. Moreover, one can use the formula 
\begin{equation}
e^{[a,b]} = \lim_k \left(e^{\frac{-a}{k}}e^{\frac{-b}{k}}e^{\frac{a}{k}}e^{\frac{b}{k}}\right)^{k^2}
\end{equation}
(a variation of the argument given in the proof of \cite[Theorem 2.11]{HallBook} will give the desired formula), together with $\theta$ being a continuous group homomorphism, to show that it respects commutation. Note that the same proof indeed shows that $S_\theta$ respects commutation, although we never explicitly used this.
From here, $G_\theta([A,A]) \subseteq [B,B]$, and consequently $G_\theta(\ov{[A,A]}) \subseteq \ov{[B,B]}$ by continuity.
Thus there is an induced $\bR$-linear map
\begin{equation}
\tilde{G}_\theta: A/\ov{[A,A]} \to B/\ov{[B,B]}. 
\end{equation}
The fact that $\tilde{G}_\theta$ is a lift of the maps on the right of (\ref{eq:gch-commute}) and (\ref{eq:gh-commute}) follows from the same arguments as in Propositions \ref{prop:S-theta-is-a-lift} and \ref{prop:S-theta-is-a-lift-H}.
\end{proof}
\end{prop}

\begin{remark}
We note that $G_\theta$ could have also been defined in the same manner as $S_\theta$. In particular, there is a correspondence
\begin{equation}\label{eq:G-theta-correspondence}
\theta(e^{2\pi i t a}) = e^{2\pi i t G_\theta(a)}, t \in \bR.
\end{equation}
One can use this to show that $\theta(\ker \Delta_A^1) \subseteq \ker \Delta_B^1$ as in Proposition \ref{prop:pre-det-commuting-diagrams} by using $G_\theta$ in place of $S_\theta$, along with the fact that $G_\theta(\ov{[A,A]}) \subseteq \ov{[B,B]}$. So there will always be a general linear variant of the commuting diagram (\ref{eq:h-commute-pre-det}). 
\end{remark}

Let us use (\ref{eq:G-theta-correspondence}) to show that the maps $G_\theta$ and $\tilde{G}_\theta$ in Proposition \ref{prop:gl-G-theta-lift} are not always $\bC$-linear.

\begin{example}\label{example:non-C-linear}
Consider $A = B = \bC$ and $\theta: \bC^\times \to \bC^\times$ given by $\theta(z) = |z|^{(\alpha + \beta i)}z^n$ where $\alpha,\beta \in \bR$ and $n \in \bZ$. It is easy to see that $\theta$ is a continuous group homomorphism. However, the map $G_\theta$ is not $\bC$-linear. Indeed, we have that
\begin{equation}
\begin{split}
\theta(e^{2\pi i t(a+bi)}) &= |e^{2\pi i t(a + bi)}|^{\alpha + \beta i}e^{2\pi i t n(a + bi)} \\
&= (e^{-2\pi t b})^{\alpha + \beta i}e^{2\pi i t n(a + bi)} \\
&= e^{-2\pi t b\alpha}e^{-2\pi i t\beta i b}e^{2\pi i t n(a + bi)} \\
&= e^{2\pi i t((na - \alpha b) + i(n - \beta)b}.
\end{split}
\end{equation}
In particular, thinking of $\bC$ as $\bR^2$ with 1 and $i$ corresponding to the basis vectors $(1,0)$ and $(0,1)$ respectively, we have that $G_\theta: \bR^2 \to \bR^2$ is the map
\begin{equation}
G_\theta\begin{pmatrix}
a \\ b
\end{pmatrix} = \begin{pmatrix}
n & -\alpha \\ 0 & n - \beta
\end{pmatrix}\begin{pmatrix}
a \\ b
\end{pmatrix}.
\end{equation}
In this example, the map $G_\theta$ is $\bC$-linear if and only if $\alpha = \beta = 0$. We do note, however, that $\theta$ sends unitaries to unitaries and $\theta|_{\bT}(z) = z^n$.
\end{example}

In general, it is clear that if $\theta: GL(A) \to GL(B)$ sends unitaries to unitaries, then we can use techniques in Section \ref{section:ugh-order-on-aff-classification} to get maps between spaces of continuous affine functions on the trace simplices. If one had that $\theta$ was the restriction of a *-homomorphism or a conjugate-linear *-homomorphism, then this would be true. 

\section{Final remarks and open questions}\label{section:ugh-final-remarks}

An alternate way to construct the map $\Lambda_\theta$, using duality of traces, is as follows. Denote by $\fT_s(A)$ the set of all tracial functionals on $A$. Suppose that $A,B$ are unital C*-algebras with $T(A),T(B) \neq \emptyset$ and $\theta: U^0(A) \to U^0(B)$ is a continuous homomorphism. Define, for $a \in A_{sa}$ and $\tau \in \fT_s(B)$,
\begin{equation}\label{eq:F-tau-def}
F(\tau)(a) := \lim_{n \to \infty} \tau\left(\frac{n}{2\pi i}\log \theta(e^{2\pi i \frac{a}{n}})\right).
\end{equation}

\begin{prop}
For $\tau \in \fT_s(B)$, the map $F(\tau): A_{sa} \to \bR$ is a well-defined, bounded, self-adjoint, tracial functional. Moreover, $F: \fT_s(B) \to \fT_s(A)$ given by $\tau \mapsto F(\tau)$ is a bounded $\bR$-linear map. 
\end{prop}
\begin{proof}
Using the same arguments as in Proposition \ref{prop:gl-G-theta-lift}, it is clear the formula (\ref{eq:F-tau-def}) is well-defined and gives rise to a bounded $\bR$-linear map $F: \fT_s(B) \to \fT_s(A)$. 
\end{proof}

One can identify $\left(A_{sa}/A_0\right)^* \simeq \fT_s(A)$ (see, for example \cite{CuntzPedersen79}), and so we can use duality to get a map $\tilde{\Lambda}_\theta:= F^*: \fT_s(A)^* \to \fT_s(B)^*$ and restrict it to the dense set $A_{sa}/A_0$. One can check that the image lies in $B_{sa}/B_0$ and that the restriction is just the map $\Lambda_\theta$ that we got before. 

We finish by listing some open problems.

\begin{enumerate}
\item There are classes where topological isomorphisms between $U(A)$ and $U(B)$ (or even $U^0(A)$ and $U^0(B)$) come from *-isomorphisms or anti-*-isomorphisms. For example, if $A,B$ are prime traceless C*-algebras containing full square zero elements, this is true by results in \cite{ChandRobert23}.

If $A$ is a unital, separable, nuclear C*-algebra satisfying the UCT and $B$ is a unital simple separable nuclear $\cZ$-stable C*-algebra, then unital embeddings $A \into B$ are classified by an invariant $\uv{K}T_u(\cdot)$ which is an enlargement of $KT_u$ \cite{CGSTW23}. Thus any isometric unitary group homomorphism $U(A) \to U(B)$ will give a $KT_u$-morphism $KT_u(\theta)$ and therefore there will be an embedding $\phi:A \into B$ such that $KT_u(\phi) = KT_u(\theta)$. However it is not clear that $\phi$ satisfies $\phi|_{U(A)} = \theta$. More generally though -- in the tracial setting -- are there continuous group homomorphisms which do not have lifts to *-homomorphisms or anti-*-homomorphisms?

Note that in \cite[Chapter 6]{AraMathieubook}, Lie isomorphisms between certain C*-algebras are shown to be the sum of a Jordan *-isomorphism and a center-valued trace. Is there a result for general (injective) Lie homomorphisms between certain classes of C*-algebras?

\item This enlargement of $KT_u$ discussed in \cite{CGSTW23} contains $K$-theory with coefficients (along with appropriate pairing maps -- the Bosckstein maps discussed in \cite{SchochetIV}). So we ask: do continuous group homo\hyp{}morphisms induce maps between $K$-theory with coefficients?

\item For a general continuous homomorphism $\theta: U^0(A) \to U^0(B)$, does the norm $\|S_\theta\|$ determine a Lipschitz constant for $\theta$? We clearly have that
\begin{equation}
\|S_\theta\| \leq \inf\{ K \mid \theta \text{ is }K\text{-Lipschitz}\}
\end{equation}
by Lemma \ref{lem:lipschitzSbounded}. Is this equality?
\item For $A$ simple (or prime), is it true that any continuous injective homo\hyp{}morphism $\theta: U^0(A) \to U^0(B)$ is isometric? Contractive? What if $B$ is simple (or prime)?

\item In the initial draft of this paper, we claimed that any continuous group homomorphism in Proposition \ref{prop:gl-G-theta-lift} gave rise to a $\bC$-linear $\tilde{G}_\theta$. This is clearly false by Example \ref{example:non-C-linear}. Is there a way to guarantee that the map $\tilde{G}_\theta$ is $\bC$-linear? Or can one alter it accordingly for this to happen? Or alter it to get a map between unitary groups, which would then allow one to use the results in Section \ref{section:ugh-order-on-aff-classification}? Maybe if one starts with an injective, contractive group homomorphism $GL^0(A) \to GL^0(B)$ which sends $\bC^\times$ to $\bC^\times$, one can say something. 
\end{enumerate}

  \bibliographystyle{amsalpha}
  \bibliography{biblio}

\end{document}

%% file: preabmle.tex
\usepackage[T1]{fontenc}
\usepackage[utf8]{inputenc}
\usepackage[pdftex]{graphicx}
\usepackage{tikz-cd}
\usepackage{blindtext}

\usepackage{hyperref,xcolor}
\usepackage{amsmath,amsfonts,amssymb,amsthm,color,mathtools,enumitem}
\usepackage{epstopdf}
\usepackage{polynom}
\usepackage{babel}
\usepackage{mathrsfs}
\usepackage{chngcntr}
\usepackage{verbatim} 

\hypersetup{
	colorlinks=true,
	linkcolor=blue,
	citecolor=cyan,
	urlcolor=cyan}
	
	\usepackage{hyphenat}

\newtheorem{theorem}{Theorem}[section]

\newtheorem{prop}[theorem]{Proposition}
\newtheorem{cor}[theorem]{Corollary}
\newtheorem{lemma}[theorem]{Lemma}
\newtheorem{example}[theorem]{Example}
\newtheorem{remark}[theorem]{Remark}

\newtheorem{result}{Theorem}

\newtheorem{resultcor}[result]{Corollary}

\numberwithin{equation}{section}

%% file: macros.tex
\newcommand{\into}{\hookrightarrow}

\newcommand{\Tr}{\text{Tr}}

\newcommand{\tD}{\tilde{\Delta}}
\newcommand{\uv}{\underline}
\newcommand{\tr}{\text{tr}}
\newcommand{\id}{\text{id}}

\newcommand{\Ad}{\text{Ad}}

\newcommand{\I}{\text{I}}

\DeclareMathOperator{\Aff}{Aff}
\DeclareMathOperator{\KT}{KT}

\newcommand{\dlim}{\underset{\to}{\lim}}

%% file: letterfonts.tex
\newcommand{\ov}{\overline}
\newcommand{\bN}{\mathbb{N}}
\newcommand{\bR}{\mathbb{R}}
\newcommand{\bZ}{\mathbb{Z}}
\newcommand{\bQ}{\mathbb{Q}}

\newcommand{\bC}{\mathbb{C}}

\newcommand{\bT}{\mathbb{T}}

\newcommand{\cS}{\mathcal{S}}

\newcommand{\cZ}{\mathcal{Z}}

\newcommand{\bh}{\mathcal{B}(\mathcal{H})}

\newcommand{\fX}{\mathfrak{X}}
\newcommand{\fY}{\mathfrak{Y}}

\newcommand{\fT}{\mathfrak{T}}

%% file: unitary_group_homs.bbl
\newcommand{\etalchar}[1]{$^{#1}$}
\providecommand{\bysame}{\leavevmode\hbox to3em{\hrulefill}\thinspace}
\providecommand{\MR}{\relax\ifhmode\unskip\space\fi MR }
\providecommand{\MRhref}[2]{%
  \href{http://www.ams.org/mathscinet-getitem?mr=#1}{#2}
}
\providecommand{\href}[2]{#2}
\begin{thebibliography}{ARBG12}

\bibitem[AM03]{AraMathieubook}
Pere Ara and Martin Mathieu, \emph{Local multipliers of {C}*-algebras},
  Springer Monographs in Mathematics, Springer-Verlag London, Ltd., London,
  2003. \MR{1940428}

\bibitem[ARBG12]{AlBoothGiordano12}
Ahmed Al-Rawashdeh, Andrew Booth, and Thierry Giordano, \emph{Unitary groups as
  a complete invariant}, J. Funct. Anal. \textbf{262} (2012), no.~11,
  4711--4730. \MR{2913684}

\bibitem[Boo98]{Booth98}
Andrew Booth, \emph{The unitary group as a complete invariant for simple unital
  {AF} algebras}, Master's thesis, University of Ottawa (Canada), 1998.

\bibitem[Bre93]{Bresar93}
Matej Bre\v{s}ar, \emph{Commuting traces of biadditive mappings,
  commutativity-preserving mappings and {L}ie mappings}, Trans. Amer. Math.
  Soc. \textbf{335} (1993), no.~2, 525--546. \MR{1069746}

\bibitem[CET{\etalchar{+}}21]{CETWW21}
Jorge Castillejos, Samuel Evington, Aaron Tikuisis, Stuart White, and Wilhelm
  Winter, \emph{Nuclear dimension of simple {C}*-algebras}, Invent. Math.
  \textbf{224} (2021), no.~1, 245--290. \MR{4228503}

\bibitem[CGS{\etalchar{+}}23]{CGSTW23}
Jos{\'e}~R. Carri{\'o}n, James Gabe, Christopher Schafhauser, Aaron Tikuisis,
  and Stuart White, \emph{Classifying *-homomorphisms {I}: Unital simple
  nuclear {C}*-algebras}, arXiv:2307.06480 (2023).

\bibitem[Con85]{Conway19}
John~B. Conway, \emph{A course in functional analysis}, Graduate Texts in
  Mathematics, vol.~96, Springer-Verlag, New York, 1985. \MR{768926}

\bibitem[CP79]{CuntzPedersen79}
Joachim Cuntz and Gert~K. Pedersen, \emph{Equivalence and traces on
  {C}*-algebras}, Journal of Functional Analysis \textbf{33} (1979), no.~2,
  135--164.

\bibitem[CR23]{ChandRobert23}
Abhinav Chand and Leonel Robert, \emph{Simplicity, bounded normal generation,
  and automatic continuity of groups of unitaries}, Adv. Math. \textbf{415}
  (2023), Paper No. 108894, 52. \MR{4543451}

\bibitem[Dad95]{Dadarlat95}
Marius Dadarlat, \emph{Reduction to dimension three of local spectra of real
  rank zero {C}*-algebras}, J. Reine Angew. Math. \textbf{460} (1995),
  189--212. \MR{1316577}

\bibitem[dlH13]{dlHarpe13}
Pierre de~la Harpe, \emph{Fuglede--{K}adison determinant: theme and
  variations}, Proc. Natl. Acad. Sci. USA \textbf{110} (2013), no.~40,
  15864--15877. \MR{3363445}

\bibitem[dlHS84a]{dlHS84a}
Pierre de~la Harpe and Georges Skandalis, \emph{D\'{e}terminant associ\'{e} \`a
  une trace sur une alg\'{e}bre de {B}anach}, Ann. Inst. Fourier (Grenoble)
  \textbf{34} (1984), no.~1, 241--260. \MR{743629}

\bibitem[dlHS84b]{dlHS84b}
\bysame, \emph{Produits finis de commutateurs dans les {C}*-alg\`ebres}, Ann.
  Inst. Fourier (Grenoble) \textbf{34} (1984), no.~4, 169--202. \MR{766279}

\bibitem[Dye53]{Dye53}
H.~A. Dye, \emph{The unitary structure in finite rings of operators}, Duke
  Math. J. \textbf{20} (1953), 55--69. \MR{52695}

\bibitem[Dye55]{Dye55}
\bysame, \emph{On the geometry of projections in certain operator algebras},
  Ann. of Math. (2) \textbf{61} (1955), 73--89. \MR{66568}

\bibitem[EGLN15]{EGLN15}
George~A. Elliott, Guihua Gong, H.~Lin, and Zhuang Niu, \emph{On the
  classification of simple amenable {C}*-algebras with finite decomposition
  rank, {II}}, arXiv:1507.03437 (2015).

\bibitem[Fol16]{Follandbook16}
Gerald~B. Folland, \emph{A course in abstract harmonic analysis}, second ed.,
  Textbooks in Mathematics, CRC Press, Boca Raton, FL, 2016. \MR{3444405}

\bibitem[GLN20a]{GongLinNiu20I}
Guihua Gong, Huaxin Lin, and Zhuang Niu, \emph{A classification of finite
  simple amenable $\mathcal{Z}$-stable {C}*-algebras, {I}: {C}*-algebras with
  generalized tracial rank one}, C. R. Math. Acad. Sci. Soc. R. Can.
  \textbf{42} (2020), no.~3, 63--450. \MR{4215379}

\bibitem[GLN20b]{GongLinNiu20II}
\bysame, \emph{A classification of finite simple amenable $\mathcal{Z}$-stable
  {C}*-algebras, {II}: {C}*-algebras with rational generalized tracial rank
  one}, C. R. Math. Acad. Sci. Soc. R. Can. \textbf{42} (2020), no.~4,
  451--539. \MR{4215380}

\bibitem[Gon97]{Gong97}
Guihua Gong, \emph{On inductive limits of matrix algebras over
  higher-dimensional spaces. {I}, {II}}, Math. Scand. \textbf{80} (1997),
  no.~1, 41--55, 56--100. \MR{1466905}

\bibitem[Goo86]{Goodearlbook}
Kenneth~R. Goodearl, \emph{Partially ordered abelian groups with
  interpolation}, Mathematical Surveys and Monographs, vol.~20, American
  Mathematical Society, Providence, RI, 1986. \MR{845783}

\bibitem[GS16]{GiordanoSierakowski16}
Thierry Giordano and Adam Sierakowski, \emph{The general linear group as a
  complete invariant for {C}*-algebras}, J. Operator Theory \textbf{76} (2016),
  no.~2, 249--269. \MR{3552377}

\bibitem[Hal15]{HallBook}
Brian Hall, \emph{Lie groups, {L}ie algebras, and representations}, second ed.,
  Graduate Texts in Mathematics, vol. 222, Springer, Cham, 2015, An elementary
  introduction. \MR{3331229}

\bibitem[HM14]{HatoriMolnar14}
Osamu Hatori and Lajos Moln\'{a}r, \emph{Isometries of the unitary groups and
  {T}hompson isometries of the spaces of invertible positive elements in
  {C}*-algebras}, J. Math. Anal. Appl. \textbf{409} (2014), no.~1, 158--167.
  \MR{3095026}

\bibitem[Jia97]{Jiang97}
Xinhui Jiang, \emph{Nonstable {K}-theory for $\mathcal{Z}$-stable
  {C}*-algebras}, arXiv preprint math/9707228 (1997).

\bibitem[Ng14]{Ng14}
Ping~W. Ng, \emph{The kernel of the determinant map on certain simple
  {C}*-algebras}, J. Operator Theory \textbf{71} (2014), no.~2, 341--379.
  \MR{3214642}

\bibitem[NR17]{NgRobert17}
Ping~W. Ng and Leonel Robert, \emph{The kernel of the determinant map on pure
  {C}*-algebras}, Houston J. Math. \textbf{43} (2017), no.~1, 139--168.
  \MR{3647937}

\bibitem[Pat83]{Paterson83}
Alan L.~T. Paterson, \emph{Harmonic analysis on unitary groups}, J. Funct.
  Anal. \textbf{53} (1983), no.~3, 203--223. \MR{724026}

\bibitem[Pau02]{Paulsenbook}
Vern Paulsen, \emph{Completely bounded maps and operator algebras}, Cambridge
  Studies in Advanced Mathematics, vol.~78, Cambridge University Press,
  Cambridge, 2002. \MR{1976867}

\bibitem[Phi92]{Phillips92}
N.~Christopher Phillips, \emph{The rectifiable metric on the space of
  projections in a {C}*-algebra}, Internat. J. Math. \textbf{3} (1992), no.~5,
  679--698. \MR{1189681}

\bibitem[Phi00]{Phillips00}
\bysame, \emph{A classification theorem for nuclear purely infinite simple
  {C}*-algebras}, Doc. Math. \textbf{5} (2000), 49--114. \MR{1745197}

\bibitem[Rie87]{Rieffel87}
Marc~A. Rieffel, \emph{The homotopy groups of the unitary groups of
  noncommutative tori}, J. Operator Theory \textbf{17} (1987), no.~2, 237--254.
  \MR{887221}

\bibitem[RLL00]{RordamKBook}
Mikael R{\o}rdam, Flemming Larsen, and Niels~J. Laustsen, \emph{An introduction
  to {K}-theory for {C}*-algebras}, London Mathematical Society Student Texts,
  vol.~49, Cambridge University Press, Cambridge, 2000. \MR{1783408}

\bibitem[R{\o}r02]{RordamBook}
Mikael R{\o}rdam, \emph{Classification of nuclear, simple {C}*-algebras},
  Classification of nuclear {C}*-algebras. {E}ntropy in operator algebras,
  Encyclopaedia Math. Sci., vol. 126, Springer, Berlin, 2002, pp.~1--145.
  \MR{1878882}

\bibitem[Sak55]{Sakai55}
Sh\^{o}ichir\^{o} Sakai, \emph{On the group isomorphism of unitary groups in
  {AW}*-algebras}, Tohoku Math. J. (2) \textbf{7} (1955), 87--95. \MR{73139}

\bibitem[Sak12]{Sakaibook}
\bysame, \emph{{C}*-algebras and {W}*-algebras}, Springer Science \& Business
  Media, 2012.

\bibitem[Sch84]{SchochetIV}
Claude Schochet, \emph{Topological methods for {C}*-algebras. {IV}. {M}od {$p$}
  homology}, Pacific J. Math. \textbf{114} (1984), no.~2, 447--468. \MR{757511}

\bibitem[Tho93]{Thomsen93}
Klaus Thomsen, \emph{Finite sums and products of commutators in inductive limit
  {C}*-algebras}, Ann. Inst. Fourier (Grenoble) \textbf{43} (1993), no.~1,
  225--249. \MR{1209702}

\bibitem[Tho95]{Thomsen95}
\bysame, \emph{Traces, unitary characters and crossed products by
  {$\mathbb{Z}$}}, Publ. Res. Inst. Math. Sci. \textbf{31} (1995), no.~6,
  1011--1029. \MR{1382564}

\bibitem[Yen56]{Yen56}
Ti~Yen, \emph{Isomorphism of unitary groups in {AW}*-algebras}, Tohoku Math. J.
  (2) \textbf{8} (1956), 275--280. \MR{89376}

\end{thebibliography}
